\documentclass[letterpaper, reqno,11pt]{article}
\usepackage[margin=1.0in]{geometry}
\usepackage{color,latexsym,amsmath,amssymb, amsthm, graphicx, enumitem}
\usepackage[percent]{overpic}

\newcommand{\RR}{\mathbb{R}}
\newcommand{\CC}{\mathbb{C}}
\newcommand{\ZZ}{\mathbb{Z}}

\newcommand{\eps}{\varepsilon}
\newcommand{\tubes}{\mathbb{T}}

\newtheorem{thm}{Theorem}
\newtheorem{lem}[thm]{Lemma}
\newtheorem{defn}[thm]{Definition}
\newtheorem{prop}[thm]{Proposition}
\newtheorem{cor}[thm]{Corollary}
\newtheorem{conj}[thm]{Conjecture}

\newtheorem*{twoEndsAss}{Assertion $\mathbf{TE}(d,a,b)$}
\newtheorem*{robustTransAss}{Assertion $\mathbf{RT}(d,a,b)$}

\newtheorem{rem}[thm]{Remark}

\newcommand{\dimBoundSymbolic}{3 + \frac{1}{600}(\sqrt{17665}-97)}
\newcommand{\dimBoundDecimal}{3.059}

\newcommand{\maxmlBoundSymbolic}{3 + \frac{2195737-13925\sqrt{17665}}{6959096}}
\newcommand{\maxmlBoundDecimal}{3.049}

\begin{document}
\pagenumbering{arabic}
\title{A Kakeya maximal function estimate in four dimensions using planebrushes}
\author{Nets Hawk Katz\thanks{California Institute of Technology, Pasadena CA, supported by NSF grant DMS 1565904, nets@caltech.edu.},\and Joshua Zahl\thanks{University of British Columbia, Vancouver BC, supported by an NSERC Discovery grant, jzahl@math.ubc.ca. MSC (2010): Primary 42B25.  Keywords: Kakeya problem}}

\maketitle

\begin{abstract}
We obtain an improved Kakeya maximal function estimate and improved Kakeya Hausdorff dimension estimate in $\mathbb{R}^4$ using a new geometric argument called the planebrush. A planebrush is a higher dimensional analogue of Wolff's hairbrush, which gives effective control on the size of Besicovitch sets when the lines through a typical point concentrate into a plane. When Besicovitch sets do not have this property, the existing trilinear estimates of Guth-Zahl can be used to bound the size of a Besicovitch set. In particular, we establish a maximal function estimate in $\mathbb{R}^4$ at dimension $\maxmlBoundDecimal$, and we prove that every Besicovitch set in $\mathbb{R}^4$ must have Hausdorff dimension at least $\dimBoundDecimal.$ 
\end{abstract}

\section{Introduction}
A Besicovitch set is a compact subset of $\mathbb{R}^n$ that contains a unit line segment pointing in every direction. The Kakeya conjecture asserts that every Besicovitch set in $\mathbb{R}^n$ must have Hausdorff dimension $n$. When $n=2$ the conjecture was resolved by Davies \cite{Davies}, while in three and higher dimensions the conjecture remains open. Additional background on the Kakeya problem can be found in the surveys \cite{KT, W2}.

The Kakeya maximal function conjecture is a slightly stronger and more technical version of the Kakeya conjecture which concerns the volume of unions of long thin ``tubes'' pointing in different directions. Stated precisely, the conjecture is as follows

\begin{conj}[Kakeya maximal function conjecture]\label{KakeyaMaximalFnConjecture}
Let $\tubes$ be a set of $\delta$-tubes in $\RR^n$ that point in $\delta$-separated directions. Then for each $d\leq n$ and each $\eps>0$, there exists a constant $C_{\eps}$ so that
\begin{equation}\label{maxmlFnBdConj}
\Big\Vert \sum_{T\in\tubes}\chi_T \Big\Vert_{d/(d-1)} \leq C_{\eps} \Big(\frac{1}{\delta}\Big)^{n/d-1+\eps}.
\end{equation}
\end{conj}

For a given value of $n$ and $d$, the bound \eqref{maxmlFnBdConj} implies that every Besicovitch set in $\RR^n$ must have Hausdorff dimension at least $d$. In particular, Conjecture \ref{KakeyaMaximalFnConjecture} implies the Kakeya conjecture. 

In this paper, we will make some partial progress towards to Kakeya maximal function conjecture in $\RR^4$. Specifically, we prove the following.

\begin{thm}\label{maximalFunctionEstimateBound}
Let $\tubes$ be a set of $\delta$-tubes in $\RR^4$ that point in $\delta$-separated directions. Then for each $\eps>0$, there exists a constant $C_{\eps}$ so that
\begin{equation}\label{maxmlFnBd}
\Big\Vert \sum_{T\in\tubes}\chi_T \Big\Vert_{d/(d-1)} \leq C_{\eps} \Big(\frac{1}{\delta}\Big)^{4/d-1+\eps},
\end{equation}
where $d=\maxmlBoundSymbolic\geq \maxmlBoundDecimal$. Furthermore,  every Besicovitch set in $\RR^4$ has Hausdorff dimension at least $\dimBoundSymbolic\geq \dimBoundDecimal$.
\end{thm}

The proof of Theorem \ref{maximalFunctionEstimateBound} involves a new geometric ingredient, which we call the ``planebrush'' argument. Recall that Wolff's ``hairbrush'' argument from \cite{W} hinges on the following geometric observation: If $T$ is a $\delta$-tube in $\RR^n$, then $\RR^n$ can be written as a union of $\sim\delta^{1-n}$ sets, each of which is the $\delta$-neighborhood of a plane containing the line coaxial with $T$; we call these sets thickened planes. Morally speaking, these thickened planes are disjoint, and if $T^\prime$ is a tube intersecting $T$ (say at angle comparable to 1), then $T^\prime$ must be contained in one of these thickened planes. C\'ordoba's two-dimensional Kakeya argument from \cite{Cor} can then be applied to the tubes contained in each thickened plane; the conclusion is that the tubes inside each thickened plane are essentially disjoint. Thus the set of tubes intersecting a fixed tube $T$ are essentially disjoint. If few tubes intersect a typical tube $T$ then the Besicovitch set must have large volume. On the other hand, if many tubes intersect a typical tube $T$, then the Besicovitch set must also have large volume, since the tubes intersecting $T$ are all disjoint and contained in the Besicovitch set. When this argument is made precise, it shows that every Besicovitch set in $\RR^n$ must have Hausdorff dimension at least $\frac{n+2}{2}$.

The planebrush argument employs a similar idea, except instead of dividing $\RR^n$ into thickened planes, all of which contain a common tube, we will divide $\RR^n$ into thickened 3-planes, all of which contain a common thickened plane. The advantage of this approach is that a larger number of tubes are contained in the resulting collection of thickened 3-planes. A disadvantage of this approach is that the Kakeya problem in $\RR^3$ remains open, so we do not have a nice analogue of C\'ordoba's argument. Despite this shortcoming, the planebrush argument can still yield superior bounds compared to Wolff's argument in certain special cases.

Theorem \ref{maximalFunctionEstimateBound} improves upon the earlier result of Guth-Zahl, Zahl, and Katz-Rogers \cite{GZ, Z, KR}, which established \eqref{maxmlFnBdConj} for $d=3+1/40=3.025$\footnote{Guth-Zahl \cite{GZ} originally claimed \eqref{maxmlFnBdConj} for $d=3+1/28$, but that proof contained an arithmetic error that has since been corrected; the correct bound established by that argument is $3+1/40.$}.

\subsection{Thanks}
The authors would like to thank Keith Rogers, Mukul Rai Choudhuri, Mingfeng Chen, and the anonymous referees for comments and corrections on an earlier version of this manuscript.

\section{Technical preliminaries and tools}

\subsection{Tubes, shadings, and refinements}
\begin{defn}
A $\delta$-tube is the $\delta$-neighborhood of a unit line segment in $\RR^4$. Every $\delta$-tube has measure $\sim\delta^3$. We say that two $\delta$-tubes are essentially identical if the 2-fold dilate of one of the tubes contains the other. If two tubes are not essentially identical then we say they are essentially distinct.
\end{defn}

\begin{defn}
A $\delta$-cube is a set of the form $Q=[0,\delta)^4+v$, where $v\in (\delta\ZZ)^4$. Observe that the set of all $\delta$-cubes tile $\RR^4$. The symbol $Q$ will always refer to a $\delta$-cube, so for example the expression $\sum_{Q\subset A}f(Q)$ will refer to a sum taken over all $\delta$-cubes contained in the set $A$.
\end{defn}

\begin{defn}\label{shading}
Let $T$ be a $\delta$-tube. A shading of $T$ is a set $Y(T)$ that is a union of $\delta$-cubes, each of which intersect $T$. Let $\tubes$ be a set of $\delta$-tubes; for each $T\in\tubes$, let $Y(T)$ be a shading of $T$. We refer to the pair $(\tubes,Y)$ as a set of tubes and their associated shading.
\end{defn}

For each $\delta$-cube $Q$, define 
\begin{equation*}
\tubes_Y(Q)=\{T\in\tubes\colon Q\subset Y(T)\}.
\end{equation*}
We will sometimes write this as $\tubes(Q)$ if the shading $Y$ is apparent from context. For each $T\in\tubes$, define 
\begin{equation*}
H_Y(T)=\bigcup_{Q\subset Y(T)}\tubes_Y(Q).
\end{equation*}
This set is called the hairbrush of $T$.

If $(\tubes,Y)$ is a set of tubes and their associated shading, define $\mathcal{Q}(Y)$ to be the set of $\delta$-cubes that are contained in at least one shading $Y(T)$ for some $T\in\tubes$. 

If $(\tubes,Y)$ is a set of tubes and their associated shading, define 
\begin{equation*}
\lambda_Y = \frac{1}{|\tubes|}\sum_{T\in\tubes}\frac{|Y(T)|}{|T|},
\end{equation*}
and define
\begin{equation*}
\mu_Y = \frac{1}{|\mathcal{Q}(Y)|}\sum_{Q\in\mathcal{Q}(Y)}|\tubes_Y(Q)|.
\end{equation*}
$\lambda_Y$ is the average shading density of a tube from $(\tubes,Y)$, and $\mu_Y$ is the average multiplicity of a cube from $\mathcal{Q}(Y)$ (i.e. the average number of tubes from $(\tubes,Y)$ whose shading contains $Q$).

\begin{defn}[Refinements]
Let $(\tubes,Y)$ be a set of tubes and their associated shading and let $t>0$. We say that a pair $(\tubes^\prime,Y^\prime)$ is a \emph{$t$-refinement} of $(\tubes,Y)$ if $\tubes^\prime\subset\tubes,$ $Y^\prime(T)\subset Y(T)$ for each $T\in\tubes^\prime$, and
\begin{equation*}
\sum_{T\in\tubes^\prime}|Y^\prime(T)|\geq t\sum_{T\in\tubes}|Y(T)|.
\end{equation*}

For example, if $(\tubes,Y)$ is a set of tubes and their associated shading, then there exists a $\sim|\log\delta|^{-1}$-refinement $(\tubes^\prime,Y^\prime)$ so that $\mu_{Y^\prime}\leq |\tubes^\prime_{Y^\prime}(Q)|\leq2\mu_{Y^\prime}$ for all $Q\in\mathcal{Q}(Y^\prime)$. Since $|\log\delta|^{-1}$-refinements will frequently occur in our proof, sometimes we will abuse notation and simply refer to them as refinements. 
\end{defn}

Observe that if $(\tubes^\prime,Y^\prime)$ is a $t$-refinement of $(\tubes,Y)$, then $\lambda_{Y^\prime}\geq t \lambda_Y$. Of course it is possible that $\lambda_{Y^\prime}$ might be much larger than $\lambda_Y$.

\subsection{Replacing sets with large homogeneous subsets}
The following lemma is an abstract formulation of the ``two-ends'' reduction that is frequently used when studying the Kakeya problem. In the following lemmas, we will apply this abstract version to several concrete situations.

\begin{lem}\label{largeHomogeneousPiece}
Let $(X,d)$ be a finite metric space of diameter at most one, and let $\eps>0$. Then there exists a point $x_0\in X$ and a radius $r_0>0$ so that the following holds.
\begin{itemize}
\item $|X\cap B(x_0,r_0)|\geq \frac{1}{2}|X|^{1-\eps}.$
\item For every $x\in X$ and every $r\geq 1/|X|$, 
\begin{equation}\label{nonConcentrationBalls}
|B(x,r) \cap X \cap B(x_0,r_0)| \leq 2(r/r_0)^{\eps}|X\cap B(x_0,r_0)|.
\end{equation}
\end{itemize}
\end{lem}
 
\begin{proof}
We will closely follow Tao's argument from \cite{T2}. Consider the quantity
\begin{equation*}
\sup_{B(x,r)}r^{-\eps}|X\cap B(x,r)|,
\end{equation*}
where the supremum is taken over all balls $B(x,r)$ with $x\in X$ and $r\geq 1/|X|$. Since $X$ is finite and $r\geq 1/|X|$, we have $r^{-\eps}|X\cap B(x,r)|\leq |X|^{1+\eps}$, so the above supremum is finite. In particular, there exists a ball $B(x_0,r_0)$ with $x_0\in X$ and $r_0\geq|X|^{-1}$ which comes within a factor of 2 of achieving the supremum. With this choice of ball, we have $r_0^{-\eps}|X\cap B(x_0,r_0)|\geq\frac{1}{2}(\operatorname{diam}(X))^{-\eps}|X|$, and thus $|X\cap B(x_0,r_0)|\geq r_0^{\eps}|X|\geq |X|^{1-\eps}$. 

Next, observe that for every $x\in X$ and every $r\geq 1/|X|$, we have
\begin{equation*}
r^{-\eps}|B(x,r) \cap X\cap B(x_0,r_0)| \leq r^{-\eps}|X\cap B(x,r)|\leq 2r_0^{-\eps}|X\cap B(x_0,r_0)|,
\end{equation*}
which establishes \eqref{nonConcentrationBalls}.
\end{proof}

\begin{lem}\label{twoEndsBuildingBlock}
Let $(X_1,d_1),\ldots,(X_n,d_n)$ be finite metric spaces of diameter at most one, each of which have cardinality at most $N$. Fix $\eps>0$. Then there is a radius $r_0>0$, a set of indices $I\subset[n]$, and a set of points $\{x_i\}_{i\in I}$ so that
\begin{itemize}
\item $|I|\gtrsim n/\log^2 N$.
\item For every $i\in I$, $x_i\in X_i$.
\item For every $i\in I$, $|X_i\cap B(x_i,r_0)|\geq \frac{1}{4}|X_i|^{1-\eps}.$
\item For every $i\in I$ and for every $x\in X_i$ and every $r>|X_i|^{-1}$, 
\begin{equation*}
|B(x,r) \cap X \cap B(x_i,r_0)| \leq 8(r/r_0)^{\eps}|X_i\cap B(x_i,r_0)|.
\end{equation*}
\end{itemize}
\end{lem}
\begin{proof}
For each $i=1,\ldots,n$, apply Lemma \ref{largeHomogeneousPiece} to $(X_i,d_i)$ with $\eps$ as above, and let $B(x_i,r_i)$ be the resulting ball. Since $\frac{1}{N}\leq r_i\leq 1$ for each index $i$, we can select a radius $r_0$, an integer $1\leq N_0\leq N$, and a set $I\subset[n]$ of cardinality $|I|\geq n/\log^2 n$ so that $\frac{1}{2}r_0\leq r_i\leq r_0$ and $N_0\leq |X_i|\leq 2N_0$ for each $i\in I$. 
\end{proof}
\subsection{The two-ends reduction}

\begin{defn}[Two-ends condition]\label{twoEndsDef}
Let $(\tubes,Y)$ be a set of $\delta$-tubes and their associated shading. We say that $(\tubes,Y)$ is $(\eps_1,C_1)$-two-ends if for all $T\in\tubes$ and all $\delta\leq r\leq 1$, we have
\begin{equation}\label{nonConcentrationInTube}
|\{Q\colon Q\subset Y(T)\cap B(x,r)\}|\leq r^{\eps_1}C_1\lambda_Y\delta^{-1}\quad\textrm{for all balls}\ B(x,r).
\end{equation}
\end{defn}

\begin{rem}
A virtue of Definition \ref{twoEndsDef} is that it is preserved under refinements: If $(\tubes,Y)$ is $(\eps_1,C_1)$-two-ends and if $(\tubes^\prime,Y^\prime)$ is a $t$-refinement of $(\tubes,Y)$, then $(\tubes^\prime,Y^\prime)$ is $(\eps_1,C_1/t)$-two-ends.
\end{rem}

The two-ends condition is valuable because collections of tubes satisfying the two-ends condition can be easier to manipulate. At the same time, Kakeya estimates about collections of tubes satisfying the two-ends condition can be upgraded to Kakeya estimates about general collections of tubes. This procedure is known as the two-ends reduction, and we will describe it below.

For $0\leq d\leq n$, $a>0$, and $0\leq b\leq 1$, define Assertion $\mathbf{TE}(d,a,b)$ to be the following statement:

\begin{twoEndsAss}
Let $\delta>0$ and let $(\tubes,Y)$ be a set of $\delta$-tubes in $\RR^4$ pointing in $\delta$-separated directions and their associated shading. Suppose that $(\tubes,Y)$ is $(\eps_1,100)$-two-ends. Then for each $\eps>0$,
\begin{equation}
\Big| \bigcup_{T\in\tubes}Y(T) \Big|\geq c\lambda_Y^{a}\delta^{4-d+\eps}(\delta^{3}|\tubes|)^b,
\end{equation}
where $c>0$ is a constant that depends only on $d,a,b,\eps,$ and $\eps_1$.
\end{twoEndsAss}

Note that if Assertion $\mathbf{TE}(d,a,b)$ is true, then Assertion $\mathbf{TE}(d,a^\prime,b)$ is also true for all $a^\prime\geq a$. The two-ends reduction says that Assertion $\mathbf{TE}(d,d,b)$ implies a Kakeya maximal function estimate at dimension $d$.
\begin{prop}[The two-ends reduction]\label{twoEndsConditionThm}
Suppose that $\mathbf{TE}(d,d,b)$ is true for some $1\leq d\leq 4$ and $0\leq b\leq 1$. Then for each $\eps>0$, there exists a constant $c_\eps$ so that the following holds. Let $(\tubes,Y)$ be a set of tubes pointing in $\delta$-separated directions and their associated shading. Then
\begin{equation}
\Big| \bigcup_{T\in\tubes}Y(T) \Big|\geq c_{\eps}\lambda_Y^{d}\delta^{4-d+\eps}(\delta^{3}|\tubes|)^b.
\end{equation}
\end{prop}
\begin{proof}
Let $(\tubes_1,Y_1)$ be a refinement of $(\tubes,Y)$ so that $|Y_1(T)|/|T|\sim\lambda_{Y_1}$ for each $T\in\tubes_1$. Write $\tubes_1 = \{T_1,\ldots,T_n\}$. For each index $i$, define $X_i$ to be the set of $\delta$-cubes $Q\subset Y_1(T)$, and define the metric $d_i$ on $X_i$, where the distance between two cubes to be the distance between their centers. 

Let $\eps_1 = \eps/(4d)$ and apply Lemma \ref{twoEndsBuildingBlock} to the metric spaces $(X_1,d_1),\ldots, (X_n,d_n)$ with this choice of $\eps_1$. We obtain a set $I\subset[n]$, a radius $r_0$, and points $\{x_i\in X_i,\ i\in I\}$. Define $\tubes_2 = \{T_i\colon i\in I\}$, and for each $T_i\in \tubes_2$, define $Y_2(T)=Y_1(T)\cap B(x_i,r)$. We have
\begin{equation*}
\sum_{T\in\tubes_2}|Y_2(T)|\gtrapprox \delta^{\eps_1}\sum_{T\in\tubes_1}|Y_1(T)|\gtrapprox \delta^{\eps_1}\sum_{T\in\tubes}|Y(T)|.
\end{equation*}

Let $\tubes_3\subset\tubes_2$ so that if we define $Y_3$ to be the restriction of $Y_2$ to $\tubes_3$, then $(\tubes_3,Y_3)$ is a refinement of $(\tubes_2,Y_2)$ with $|Y_3(T)|/|T|\sim \lambda_{Y_3}$ for each $T\in\tubes_3$. Note that
\begin{equation}
\lambda_{Y_3} = \frac{1}{|\tubes_3|}\sum_{T\in\tubes_3}|Y_3(T)|\gtrapprox \delta^{-\eps_1} \lambda_Y |\tubes|/|\tubes_3|.
\end{equation}

Let $\mathcal{B}$ be a set of finitely overlapping balls of radius $2r$ so that every ball of radius $r$ is contained in a ball from $\mathcal{B}$. After pigeonholing, we can select a set $\mathcal{B}_1\subset\mathcal{B}$ of disjoint balls so that 
\begin{equation}
\sum_{B\in\mathcal{B}_1}|\{T\in\tubes_3\colon Y_3(T)\subset B\}| \gtrsim |\tubes_3|.
\end{equation}
For each $B\in\mathcal{B}_1,$ let $\tubes_B = \{T\in\tubes_3\colon Y_3(T)\subset B\}.$ Observe that for each $B\in\mathcal{B}_1$ and each $T\in\tubes_B$, we have
\begin{equation*}
|Y_3(T)|/|T|\sim \lambda_{Y_3} \sim\frac{1}{|\tubes_B|}\sum_{T\in\tubes_B}|Y_3(T)|/|T|.
\end{equation*}

For each $B\in\mathcal{B}_1$, let $A\colon\RR^3\to\RR^3$ be an affine map that sends $B$ to the unit ball. For each $T\in\tubes_B$, let $\tilde T$ be the image of $T\cap B$ under $A$, and let $\tilde Y_3(T)$ be the image of $Y_3(T)$ (which is the same as the image of $Y_3(T)\cap B$). Note that $\lambda_{\tilde Y_3}\sim \lambda_{Y_3}/r$. Define $\tilde\delta = \delta/r$ and $\tilde \tubes_B=\{\tilde T\colon T\in\tubes_B\}$. Let $\tilde\tubes_B^\prime$ be a maximal subset of $\tilde\tubes_B$ that point in $\delta/r$-separated directions; we have 
\begin{equation*}
|\tilde\tubes_B^\prime|\gtrsim r^3|\tilde\tubes_B|=r^3|\tubes_B|,
\end{equation*}
and $(\tilde\tubes_B, \tilde Y_3)$ is $(\eps_1, 100)$ two-ends. Applying Assertion $\mathbf{TE}(d,a,b)$, we conclude that there exists a constant $c^\prime>0$ so that
\begin{equation*}
\Big|\bigcup_{T\in\tilde\tubes_B^\prime}\tilde Y_3(T)\Big|\geq c^\prime \tilde\delta^{\eps/2}\lambda_{\tilde Y_3}^d\tilde\delta^d (\tilde\delta^3|\tilde\tubes_B^\prime|)^b.
\end{equation*}
Since $\tilde\tubes_B^\prime\subset \tilde\tubes_B$, we have $\Big|\bigcup_{T\in\tilde\tubes_B^\prime}\tilde Y_3(T)\Big|\leq \Big|\bigcup_{T\in\tilde\tubes_B}\tilde Y_3(T)\Big|$. Undoing the linear transformation $A\colon\RR^3\to\RR^3$, we conclude that
\begin{equation}
\begin{split}
\Big|\bigcup_{T\in\tubes_B}Y_3(T)\Big|& \gtrsim r^4 c^\prime \tilde\delta^{\eps/2}\lambda_{\tilde Y_3}^d\tilde\delta^{4-d} (\tilde\delta^3|\tilde\tubes_B^\prime|)^b\\
&\gtrsim r^4 c^\prime(\delta/r)^{\eps/2}(\lambda_{Y_3}/r)^d(\delta/r)^{4-d} \big( (\delta/r)^3 (r^3|\tubes_B|)\big))^b\\
&\gtrapprox c^\prime \delta^{\eps/2}\lambda_{Y_3}^d \delta^{4-d}(\delta^3|\tubes_B|)^b.
\end{split}
\end{equation}
Summing over all $B\in\mathcal{B}_1$ and noting that $b\leq 1$, we have
\begin{equation}
\begin{split}
\Big|\bigcup_{T\in\tubes}Y(T)\Big|&\geq \sum_{B\in\mathcal{B}_1}\Big|\sum_{T\in\tubes_B}Y_3(T)\Big|\\
&\gtrapprox \sum_{B\in\mathcal{B}_1} c^\prime\delta^{\eps/2}\lambda_{Y_3}^d \delta^{4-d}(\delta^3|\tubes_B|)^b\\
&\gtrapprox c^\prime\delta^{\eps/2}\lambda_{Y_3}^d \delta^{4-d}\big(\delta^3\sum_{B\in\mathcal{B}_1}|\tubes_B|\big)^b\\
&\gtrapprox c^\prime\delta^{\eps/2}\big(\delta^{\eps_1}\lambda_Y |\tubes|/|\tubes_3|\big)^d \delta^{4-d}(\delta^3|\tubes_3|)^b\\
&= c^\prime\delta^{\eps/2 + \eps_1d}\lambda_Y^d\delta^{4-d} (\delta^3|\tubes|)^b \big(|\tubes|/|\tubes_3|\big)^{d-b}\\
&=c^\prime\delta^{(3/4)\eps}\lambda_Y^d\delta^{4-d} (\delta^3|\tubes|)^b.
\end{split}
\end{equation}
On the last line we used the fact that $d\geq 1$ and $b\leq 1$, so $d-b\geq 0$ (and $|\tubes|/|\tubes_3|\leq 1$). Thus if we select $c>0$ sufficiently small, we have
\begin{equation*}
\Big|\bigcup_{T\in\tubes}Y(T)\Big|\geq c\delta^{\eps}\lambda_Y^d\delta^{4-d} (\delta^3|\tubes|)^b.\qedhere
\end{equation*}
\end{proof}

\begin{rem}
For simplicity, we stated and proved the two-ends reduction for collections of tubes in $\RR^4.$ However, the same proof works in all dimensions. Note as well that the two-ends reduction also works for collections of tubes that satisfy the Wolff axioms. The only thing to verify is that a suitable choice for the refined set $\tilde\tubes_B^\prime$ in the above proof exists, and satisfies the Wolff axioms; this is easy to construct (for example, select each tube at random with probability $r^{d-1},$ where $d$ is the dimension). This observation will be useful in Appendix \ref{R3KakeyaSec} below.
\end{rem}

\subsection{The robust transversality reduction}

\begin{defn}[Robust transversality]\label{robustTransDef}
Let $(\tubes,Y)$ be a set of tubes and their associated shading. We say that $(\tubes,Y)$ is $(\eps_2,C_2)$-robustly transverse if for all $\delta$-cubes $Q$, all vectors $v$, and all $\delta\leq 1$, we have the following bound on the number of tubes passing through $Q$ and making small angle $r$ with $v$,
for any $r$ with $\delta < r <1$.
\begin{equation}\label{notTooConcentratedAngle}
|\{T\in\tubes_Y(Q)\colon \angle(v(T),v)\leq r\}|\leq r^{\eps_2}C_2\mu_Y.
\end{equation}
\end{defn}

\begin{robustTransAss}
Let $\delta>0$ and let $(\tubes,Y)$ be a set of $\delta$-tubes in $\RR^4$ pointing in $\delta$-separated directions and their associated shading. Suppose that $(\tubes,Y)$ is $(\eps_1,C_1)$-two-ends and $(\eps_2,100)$-robustly transverse. Then for each $\eps>0$,
\begin{equation}\label{twoEndsConditionRequirement}
\Big| \bigcup_{T\in\tubes}Y(T) \Big|\geq c\lambda_Y^{a}C_1^{-C/\eps_1}\delta^{4-d+\eps}(\delta^{3}|\tubes|)^b,
\end{equation}
where $C=C(d,a,b)$ is a constant that depends only on $d$, $a$, and $b$, and $c>0$ is a constant that depends only on $d,a,b,\eps,\eps_1,$ and $\eps_2$ (in particular, both constants are independent of $\delta$).
\end{robustTransAss}

Note that if Assertion $\mathbf{RT}(d,a,b)$ is true, then Assertion $\mathbf{RT}(d,a,b^\prime)$ is also true for all $b^\prime\geq b$.

\begin{prop}\label{robustTransReduction}
Suppose that Assertion $\mathbf{RT}(d,a,b)$ is true for some $1\leq d\leq 4$, $a\geq 0$, and $\frac{d-1}{3}\leq b\leq 1$. Then Assertion $\mathbf{TE}(d,a,b)$ is true.
\end{prop}

\begin{proof}
Fix $\eps>0$ and let $(\tubes,Y)$ be a set of $\delta$-tubes and their associated shading that is $(\eps_1,C_1)$-two-ends. Let $(\tubes_1,Y_1)$ be a refinement of $(\tubes,Y)$ so that $|\tubes_1(Q)|\sim \mu_{Y_1}$ for all $Q\in\mathcal{Q}(Y_1)$. Since $(\tubes,Y)$ is $(\eps_1,C_1)$-two-ends, we also have $|Y(T)|/|T|\leq C_1\lambda_Y$ for each $T\in\tubes$, and thus 
 \begin{equation}\label{tubes1Big}
|\tubes_1|\gtrapprox C_1^{1/\eps_1}|\tubes|.
\end{equation}

Write $\mathcal{Q}(Y_1) = \{Q_1,\ldots,Q_n\}$. For each index $i$, define $X_i = \{v(T)\colon T\in \tubes_1(Q_i)\}$. We identify $X_i$ with a subset of $S^3$, and let $d_i$ be the metric on $X_i$ induced by the usual metric on $S^3$.

Let $\eps_2>0$ be sufficiently small so that
\begin{equation}\label{boundOnEps2}
5\eps_2+C\eps_2/\eps_1\leq\eps/4,
\end{equation}
where $C=C(d,a,b)$ is the constant from \eqref{twoEndsConditionRequirement}.

Apply lemma \ref{twoEndsBuildingBlock} to the metric spaces $(X_1,d_1),\ldots, (X_n,d_n)$. We obtain a number $\delta\leq r\leq 1$; a set $I\subset[n]$, and vectors $\{v_i\in X_i,\ i\in I\}$. Define $\mathcal{Q}^\prime = \{Q_i\colon i\in I\}$; define $\tubes_2=\tubes_1$; and for each $T\in \tubes_2$, define $Y_2(T)$ to be the union of cubes 
\begin{equation*}
\{Q\subset Y_1(T)\colon Q \in\mathcal{Q}^\prime,\ \angle(v(T),v_Q)\leq 2r\}
\end{equation*}
(here $v_Q$ is the vector $v_i\in X_i$ described above, where $i$ is the index corresponding to $Q$).  We have
\begin{equation*}
\sum_{T\in\tubes_2}|Y_2(T)|\gtrapprox\delta^{\eps_2}\sum_{T\in\tubes_1}|Y_1(T)|\gtrapprox \delta^{\eps_2}\sum_{T\in\tubes}|Y(T)|.
\end{equation*}

Let $\mathcal{S}$ be a set of essentially distinct $10r$-tubes so that whenever two tubes $T,T^\prime\in\tubes_2$ satisfy $Y_2(T)\cap Y_2(T^\prime)\neq\emptyset$, there exists a $10r$ tube from $\mathcal{S}$ that contains $T$ and $T^\prime$. Since the tubes in $\mathcal{S}$ are essentially distinct, each tube $T\in\tubes_2$ is contained in $O(1)$ tubes from $\mathcal{S}$. By pigeonholing, we can select a set $\mathcal{S}^\prime\subset\mathcal{S}$ so that the sets 
\begin{equation*}
\{T\in\tubes_2\colon T\in S\}_{S\in\mathcal{S}^\prime}
\end{equation*}
are disjoint, 
\begin{equation*}
\sum_{S\in \mathcal{S}^\prime}|\{T\in\tubes_2\colon T\subset S\}| \gtrapprox |\tubes_2|,
\end{equation*}
and 
\begin{equation*}
\lambda_S:=\frac{1}{|\{T\in\tubes_2\colon T\subset S\}|}\sum_{\substack{T\in\tubes_{2}\\ T\subset S}}|Y_2(T)|/|T|\gtrapprox \delta^{\eps_2}\lambda\quad\textrm{for each}\ S\in\mathcal{S}^\prime.
\end{equation*}
 For each $S\in\mathcal{S}^\prime$, let $F$ be a linear map sending $S$ to a set $F(S)$ that contains $B(0,1/2)$ and is contained in $B(0,2)$. Define $\tilde\delta=\delta/s$, define 
\begin{equation*}
 \tilde\tubes_S = \{F(T)\colon T\in\tubes_2,\ T\subset S\},
\end{equation*}
 and define $\tilde Y_S(\tilde T) = F(Y_2(T))$. Note that the sets in $\tilde{\tubes}_S$ are not technically $\tilde\delta$ tubes, but they contain a $O(1)$-dilate of a $\tilde\delta$ tube and are contained in a $O(1)$ dilate of a $\tilde\delta$ tube. We will abuse notation slightly and refer to them as $\tilde\delta$ tubes. Similarly, $\tilde{Y}_S(T)$ is not a union of $\tilde\delta$ cubes. However, this minor technicality will not affect our estimates below. 

For each $T\in\tilde\tubes_S$ we have $|\tilde Y_S(\tilde T)|/|\tilde T|\sim |Y_2(T)|/|T|,$ and thus
\begin{equation}
\frac{1}{|\tilde \tubes_S|}\sum_{\tilde T \in \tilde \tubes_S }|\tilde Y_S(\tilde T)|/|\tilde T|\gtrapprox \delta^{\eps_2}\lambda_Y.
\end{equation}

Observe that for each $S\in\mathcal{S}^\prime$, we have that $(\tilde\tubes_S,\tilde Y_S)$ is $(\eps_1,\delta^{-\eps_2}C_1)$-two-ends and $(\eps_2,100)$-robustly transverse.

By Assertion $\mathbf{RT}(d,a,b)$, there exists a constant $c^\prime>0$ so that
\begin{equation*}
\Big|\bigcup_{\tilde T\in \tilde\tubes_S}\tilde Y_S(\tilde T)\Big|\geq c^\prime (C_1\delta^{-\eps_2})^{-C/\eps_1}\tilde\delta^{\eps/2}\big(\delta^{\eps_2}\lambda_Y\big)^a\tilde\delta^{4-d}(\tilde\delta^3|\tilde\tubes_S|)^{b},
\end{equation*}
and thus
\begin{equation}
\begin{split}
\Big|\bigcup_{\substack{T\in\tubes_2\\ T\subset S}} Y_2(T)\Big|
& \gtrapprox c^\prime\delta^{\eps/2+a\eps_2+C\eps_2/\eps_1}r^3\lambda_Y^a(\delta/r)^{4-d}\big( (\delta/r)^3|\{T\in\tubes_2\colon T\subset S\}|\big)^b\\
&\gtrapprox c^\prime\delta^{\eps/2+a\eps_2+C\eps_2/\eps_1}\lambda_Y^a \delta^{4-d} \big( \delta^3|\{T\in\tubes_2\colon T\subset S\}|\big)^{\frac{4-d}{3}},
\end{split}
\end{equation}
where on the second line we used the fact that $r\leq1$ and $b\geq\frac{d-1}{3}$, so $r^{d-1-3b}\geq 1$. Summing over $S\in\mathcal{S}^\prime$ and using \eqref{tubes1Big} and the fact that $b\leq 1$ and $a\leq 4$, we conclude
\begin{equation}
\begin{split}
\Big|\bigcup_{T\in\tubes_2} Y_2(T)\Big|& \gtrapprox c^\prime\delta^{\eps/2+4\eps_2+C\eps_2/\eps_1} \lambda_Y^a \delta^{4-d}( \delta^3 |\tubes_2|)^{b}\\
&\gtrapprox  c^\prime\delta^{\eps/2+5\eps_2+C\eps_2/\eps_1}\lambda_Y^a \delta^{4-d}( \delta^3 |\tubes|)^b\\
&\gtrapprox c^\prime\delta^{(3/4)\eps}\lambda_Y^a \delta^{4-d}( \delta^3 |\tubes|)^b,
\end{split}
\end{equation}
where on the last line we used \eqref{boundOnEps2}. Thus If $c>0$ is selected sufficiently small, then 
\begin{equation*}
\Big|\bigcup_{T\in\tubes} Y(T)\Big| \geq c\delta^{\eps} C_1^{-C/\eps_1}\lambda_Y^a \delta^{4-d}( \delta^3 |\tubes|)^b.\qedhere
\end{equation*}
\end{proof}

\section{Previous Kakeya-type estimates in $\RR^4$}

\subsection{Wolff's hairbrush estimate}
In \cite{W}, Wolff proved new Kakeya maximal function estimates using a geometric argument called the ``hairbrush'' argument. We will recall a consequence of Wolff's hairbrush argument. The formulation presented here is described in \cite{T3}, and our proof below  follows the arguments from \cite{T3}.

\begin{thm}\label{wolffThm}
Assertion $\mathbf{RT}(3,2,1/2)$ is true.
\end{thm}
\begin{proof}
Let $\eps_1>0,\eps_2>0$, and $\eps>0$. We need to prove that there exists an absolute constant $C$ and a constant $c=c(\eps,\eps_1,\eps_2)>0$ so that whenever $(\tubes,Y)$ is a set of direction-separated $\delta$-tubes that is $(\eps_1,C_1)$-two-ends and $(\eps_2,100)$-robustly transverse, we have
\begin{equation}\label{WolffBoundRobustTrans}
\Big|\bigcup_{T\in\tubes}Y(T)\Big|\geq c \lambda_Y^{2} C_1^{-C/\eps_1} \delta^{1+\eps} (\delta^3|\tubes|)^{1/2}.
\end{equation}

Let $(\tubes,Y)$ be a set of direction-separated $\delta$-tubes that is $(\eps_1,C_1)$-two-ends and $(\eps_2,100)$-robustly transverse. Without loss of generality, we can assume that $|\tubes_Y(Q)|\sim \mu_Y$ for each cube $Q\in\mathcal{Q}(Y)$; indeed, since $(\tubes,Y)$ is $(\eps_2,100)$-robustly transverse, this additional assumption can always be obtained after a harmless $\sim 1$-refinement of $(\tubes,Y)$. 

By Cauchy-Schwarz, we have
\begin{equation*}
|\{(T,Q,Q^\prime)\colon T\in \tubes,\ Q,Q^\prime\subset Y(T)\}|\gtrsim \delta^{-2}\lambda_Y^2 |\tubes|.
\end{equation*}
On the other hand, since $(\tubes,Y)$ is $(\eps_1,C_1)$-two-ends, we have that if the constant $c_1$ is chosen sufficiently small then
\begin{equation*}
|\{(T,Q,Q^\prime)\colon T\in \tubes,\ Q,Q^\prime\subset Y(T),\ \operatorname{dist}(Q,Q^\prime)\leq c_1 /C_1^{1/\eps_1}\}|\leq \frac{1}{2}|\{(T,Q,Q^\prime)\colon T\in \tubes,\ Q,Q^\prime\subset Y(T)\}|.
\end{equation*}
Fixing such a $c_1$, we have
\begin{equation*}
|\{(T,Q,Q^\prime)\colon T\in \tubes,\ Q,Q^\prime\subset Y(T),\ \operatorname{dist}(Q,Q^\prime)> c_1 /C_1^{1/\eps_1}\}|\gtrsim \delta^{-2}\lambda_Y^2 |\tubes|.
\end{equation*}
Since $(\tubes,Y)$ is $(\eps_2,100)$-robustly transverse, and $|\tubes_Y(Q)|\sim\mu_Y$ for each cube $Q\in\mathcal{Q}(Y)$, we have that for each of the triples $(T,Q,Q^\prime)$ described above, there is a small constant $c_2>0$ (depending on $\eps_2$) so that there are $\gtrsim \mu_Y$ tubes $T^\prime\in\tubes_Y(Q^\prime)$ with $\angle(v(T),v(T^\prime))\geq c_2$. In particular, if we define
\begin{equation*}
\mathcal{A} = \{(T,Q,T^\prime,Q^\prime)\colon T\in\tubes,\ Q,Q^\prime\subset Y(T),\ \operatorname{dist}(Q,Q^\prime)> c_1 /C_1^{1/\eps_1},\ \angle(v(T),v(T^\prime))\geq c_2\},
\end{equation*}
then
\begin{equation*}
|\mathcal{A}|\gtrsim \delta^{-2}\lambda_Y^2\mu_Y |\tubes|.
\end{equation*}
Thus by pigeonholing, there is a tube $T_0\in\tubes$ so that there are $\gtrsim \delta^{-2}\lambda_Y^2\mu_Y$ triples $(T,Q,Q^\prime)$ with the property that $(T,Q,T_0,Q^\prime)\in\mathcal{A}$. With this choice of $T_0$ fixed, define 
\begin{equation*}
\mathcal{B} = \{(T,Q,Q^\prime)\colon (T,Q,T_0,Q^\prime)\in\mathcal{A}\}.
\end{equation*}
Observe that if $T\in\tubes$ and $Q\subset Y(T)$, then there are $\lesssim 1$ cubes $Q^\prime$ so that $(T,Q,Q^\prime)\in\mathcal{A}.$ this is because any such tube must intersect $T\cap T_0$, and $\angle(T,T_0)\geq c_2$. Define
\begin{equation*}
\mathcal{C} = \{(T,Q)\colon (T,Q,Q^\prime)\in\mathcal{B}\ \textrm{for some cube}\ Q^\prime\}.
\end{equation*}
We have $|\mathcal{C}|\gtrsim \delta^{-2}\lambda_Y^2\mu_Y$. Define
\begin{equation*}
\mathcal{D} = \{Q\colon (T,Q)\in\mathcal{C}\ \textrm{for some tube}\ T\}.
\end{equation*}
We claim that 
\begin{equation}
|\mathcal{D}|\gtrsim \delta^{-2}|\log\delta|^{-1}\lambda_Y^2\mu_Y.
\end{equation}
Indeed, this follows from Wolff's hairbrush argument, which we will briefly recap here. Let $\Pi_i,\ i=1,\ldots,\delta^{-2}$ be a set of planes in $\RR^4$ with the following three properties.
\begin{itemize}
\item Each plane $\Pi_i$ contains the line $L_0$ coaxial with $T_0$.
\item Every $\delta$-tube intersecting $T_0$ is contained in the $2\delta$ neighborhood of at least one of these planes.
\item The $2\delta$-neighborhoods of the planes are boundedly overlapping far from $L_0$. More precisely, for each $t>\delta$ the sets $\{N_{2\delta}(\Pi_i)\backslash N_{t}(L_0) \}$ are $\lesssim 1+1/t$ overlapping. 
\end{itemize}
Note that there is a number $t \gtrsim c_1c_2 /C_1^{1/\eps_1}$ so that $\operatorname{dist}(Q,L_0)\geq t$ for all $Q\in\mathcal{D}$. Thus if we define 
\begin{equation*}
\mathcal{C}_i = \{(T,Q)\in \mathcal{C}\colon T\subset N_{2\delta}(\Pi_i) \},
\end{equation*}
and
\begin{equation*}
\mathcal{D}_i = \{Q\colon (T,Q)\in\mathcal{C}_i\ \textrm{for some tube}\ T\},
\end{equation*}
then $\mathcal{C} = \bigcup_{i}\mathcal{C}_i,$ $\mathcal{D} = \bigcup_{i}\mathcal{D}_i,$  and
\begin{equation}\label{CQsplit}
\Big|\bigcup_{Q\in\mathcal{C}}Q\Big| \gtrsim C_1^{-1/\eps_1} \sum_{i}\Big|\bigcup_{Q\in\mathcal{C}_i}Q\Big|,
\end{equation}
where the implicit constant depends on $c_1$ and $c_2$. Thus our task is to estimate $\Big|\bigcup_{\mathcal{C}_i}Q\Big|$ for each index $i$. 

Define 
\begin{equation*}
\tubes_i = \{T\colon (T,Q)\in\mathcal{C}_i\ \textrm{for some cube}\ Q\}.
\end{equation*}
For each $T\in\tubes_i$, define the shading $Y_i(T)$ to be the union of those cubes $Q\in\mathcal{D}_i$ with $(T,Q)\in\mathcal{C}_i$. Since each cube has volume $\delta^4$, we have
\begin{equation*}
\sum_i \sum_{T\in\tubes_i}|Y_i(T)|\geq \delta^4 |\mathcal{C}|\gtrsim \delta^{2}\lambda_Y^2\mu_Y.
\end{equation*}

By dyadic pigeonholing, we can select a set of indices $I\subset\{1,\ldots,\delta^{-2}\}$; for each index $i\in I$, there is a set of tubes $\tubes_i^\prime\subset\tubes_i$, and for each $T\in\tubes_i$, a sub-shading $Y_i^\prime(T)\subset Y_i(T)$, so that each of the sets of tubes $T_i,\ i\in I$ contains the same number of tubes (call this number $N$), and each shading $Y_i^\prime(T)$ contains the same number of cubes (call this number $M$). Since each cube has volume $\delta^4$, we have 
\begin{equation}\label{boundOnINM}
|I|\delta^4 NM =  \sum_{i\in I} \sum_{T\in\tubes_i^\prime}|Y^\prime_i(T)|\geq |\log\delta|^{-3}\delta^{2}\lambda_Y^2\mu_Y.
\end{equation}
Inequality \eqref{boundOnINM} lower bounds the average number of cubes contained in each shading $Y^\prime_i(T)$. Thus by Cauchy-Schwarz, we have
\begin{equation}\label{lowerBoundM}
M\geq |\log\delta|^{-3}\delta^{-1}\lambda_Y.
\end{equation}

Each tube in $\tubes_i^\prime$ is contained in $N_{2\delta}(\Pi_i)$, so in particular C\'ordoba's two-dimensional Kakeya argument from \cite{Cor} implies that
\begin{equation*}
\begin{split}
\int_{\RR^4} \Big(\sum_{T\in\tubes_i}\chi_{Y_i^\prime(T)}\Big)^2 &  \lesssim |\log\delta| \delta^3|\tubes_i^\prime|\\
&=|\log\delta| \delta^3N.
\end{split}
\end{equation*}
By Cauchy-Schwarz, we conclude that
\begin{equation*}
\begin{split}
\Big|\bigcup_{Q\in  \mathcal{C}_i}Q\Big| & \geq \Big|\bigcup_{T\in\tubes_i^\prime}Y_i^\prime(T)\Big|\\
& \gtrsim \Big(\sum_{T\in\tubes_i^\prime}|Y_i^\prime(T)|\Big)^2 \big/ \Big(|\log\delta| \delta^3N\Big)\\
& =|\log\delta|^{-1} \delta^5 N M^2.
\end{split}
\end{equation*}
Summing in $i$, we have
\begin{equation}\label{sumICQi}
\sum_{i}\Big|\bigcup_{Q\in \mathcal{C}_i}Q\Big|\gtrsim |I| |\log\delta|^{-1} \delta^5 N M^2.
\end{equation}
Combining \eqref{CQsplit}, \eqref{boundOnINM}, \eqref{lowerBoundM}, and \eqref{sumICQi}, we conclude that
\begin{equation}
\begin{split}
\Big|\bigcup_{T\in\tubes}Y(T)\Big|&\geq \Big|\bigcup_{Q\in \mathcal{C}}Q\Big|\\
&\gtrsim C_1^{-1/\eps_1} |\log\delta|^{-7} \lambda_Y^3\mu_Y \delta^2.
\end{split}
\end{equation}
On the other hand, we have 
\begin{equation*}
\mu_Y \Big|\bigcup_{T\in\tubes}Y(T)\Big| = \sum_{T\in\tubes}|Y(T)|=\lambda_Y (\delta^3|\tubes|).
\end{equation*}
Thus
\begin{equation*}
\Big|\bigcup_{T\in\tubes}Y(T)\Big|^2 \gtrsim C_1^{-1/\eps_1} |\log\delta|^{-7} \lambda_Y^4  \delta^2 (\delta^3|\tubes|),
\end{equation*}
which implies \eqref{WolffBoundRobustTrans}.
\end{proof}

\begin{rem}
Observe that the power of $\lambda$ and of $(\delta^3|\tubes|)$ in \eqref{WolffBoundRobustTrans} is better than one would expect from a Kakeya maximal function estimate in $\RR^4$ at dimension $3$ (indeed, one would expect a bound of the form $\lambda^3\delta(\delta^3|\tubes|)^{2/3}$. This improved dependence on $\lambda$ and the cardinality of $\tubes$ is possible because $(\tubes,Y)$ satisfies the two-ends and robust transversality conditions. One must be careful with estimates such as \eqref{WolffBoundRobustTrans} because they are not preserved under the two-ends and/or robust transversality reduction described in Propositions \ref{twoEndsConditionThm} and \ref{robustTransReduction}. 

This superior dependence on $\lambda$ and the cardinality of $\tubes$ will be crucial for our arguments, because later in the proof we will prove an estimate that is similar to \eqref{WolffBoundRobustTrans}, except it will have better dependence on $\delta$ and worse dependence on $\lambda$ and the cardinality of $\tubes$; we will then interpolate this estimate with \eqref{WolffBoundRobustTrans}.
\end{rem}

\subsection{\L{}aba-Tao's X-ray estimate}
We recall Theorem 1.2 from \cite{LT}:
\begin{thm}\label{LabaTaoTheorem}
Let $\tubes$ be a set of essentially distinct $\delta$-tubes in $\RR^4$. Suppose that for each vector $v$, there are at most $m$ tubes from $\tubes$ with $\angle(v,v(T))\leq\delta$. Then for each $\eps>0$, there is a constant $C_\eps$ so that
\begin{equation}
\Big\Vert \sum_{T\in\tubes} \chi_T \Big\Vert_{3/2}\leq C_{\eps} \delta^{-\frac{1}{3}-\eps}m^{\frac{5}{36}}(\delta^3|\tubes|)^{\frac{7}{9}}.
\end{equation}
\end{thm}

It will be more convenient for us to phrase this result as a multiplicity bound on unions of tubes.

\begin{cor}\label{KTCor}
Let $(\tubes,Y)$ be a set of essentially distinct $\delta$-tubes and their associated shading, and let $\Omega\subset S^3$ be a $\delta$-separated set. Define $m = |\tubes|/|\Omega|$. Suppose that each tube from $\tubes$ points in a direction from $\Omega$, and $\sim m$ tubes from $\tubes$ point in each direction from $\Omega$. Suppose $|\tubes_Y(Q)|\leq 10\mu_Y$ for all $Q\in\mathcal{Q}(Y)$. Then 
\begin{equation}\label{muYBound}
\mu_Y \leq C_{\eps} \lambda^{-2} \delta^{-1-\eps}m^{3/4}(\delta^3|\Omega|)^{1/3}.
\end{equation}
\end{cor}
\begin{proof}
Since $|\tubes_Y(Q)|\leq 10\mu_Y$ for all $Q\in\mathcal{Q}(Y)$, there exists a set $X \subset \bigcup_{T\in\tubes}Y(T)$ that is a union of $\delta$-cubes, with $|X|\mu_Y\gtrsim \lambda_Y(\delta^3|\tubes|)$ so that $|\tubes_Y(Q)|\gtrsim \mu_Y$ for each cube $Q\subset X$. Applying Theorem \ref{LabaTaoTheorem} with $\eps/2$ in place of $\eps$, we conclude that there exists a constant $C=C(\eps)$ so that
\begin{equation*}
|X|^{2/3}\mu_Y \sim \Big\Vert \sum_{T\in\tubes} \chi_T \Big\Vert_{3/2}\leq C \delta^{-\frac{1}{3}-\eps/3}m^{\frac{5}{36}}(\delta^3|\tubes|)^{\frac{7}{9}},
\end{equation*}
and thus
\begin{equation*}
\begin{split}
\mu_Y &\leq C^3 \delta^{-1-\eps}m^{\frac{5}{12}}(\delta^3|\tubes|)^{\frac{1}{3}}\\
&= C^3 \delta^{-1-\eps}m^{\frac{3}{4}}(\delta^3|\Omega|)^{\frac{1}{3}}.\qedhere
\end{split}
\end{equation*}
\end{proof}

\subsection{Guth-Zahl's Trilinear Kakeya estimate}
In \cite{GZ}, Guth and the second author proved a trilinear Kakeya-type estimate for collections of $\delta$-tubes satisfying the generalized Wolff axioms. A set of tubes is said to satisfy the Wolff axioms if not too many tubes from the set can concentrate into the thickened neighborhood of an affine subspace. A set of tubes is said to satisfy the generalized Wolff axioms if not too many tubes from the set can concentrate into the thickened neighborhood of an algebraic variety, or more generally a semi-algebraic set. In \cite{Z} and \cite{KR}, the second author, and independently Rogers and the first author showed that collections of direction-separated tubes satisfy these requirements\footnote{In \cite{Z}, the second author proved a slightly weaker statement, but this statement is nonetheless sufficient for what follows.}. By combining the results of \cite{GZ} with \cite{KR} or \cite{Z}, we obtain the following trilinear Kakeya-type bound.

\begin{thm}[Trilinear Kakeya in $\RR^4$]\label{GZTrilinearProp}
For each $\epsilon>0$, there exists a constant $C_\epsilon$ so that the following holds. Let $\tubes$ be a set of $\delta$-tubes in $\RR^4$ that point in $\delta$-separated directions. Then
\begin{equation}\label{trilinearBoundFromGZ}
\int\Big(\sum_{T_1,T_2,T_3\in\tubes}\chi_{T_1}\ \chi_{T_2}\ \chi_{T_3}\ |v_1\wedge v_2\wedge v_3|^{12/13}\Big)^{13/27} \leq C_{\epsilon}\delta^{-1/3-\epsilon}(\delta^3|\tubes|)^{4/3},
\end{equation}
where in the above expression $v_i$ is the direction of the tube $T_i$.
\end{thm}

\begin{cor}\label{trilinearVolumeBound}
For each $\epsilon>0$, there exists a constant $C_\epsilon$ so that the following holds. Let $(\tubes,Y)$ be a set of $\delta$-tubes in $\RR^4$ that point in $\delta$-separated directions and their associated shading. Suppose that for each $Q\in\mathcal{Q}_Y(T)$, we have 
\begin{equation*}
|\{(T_1,T_2,T_3)\in\tubes_Y(Q)\colon |v_1\wedge v_2\wedge v_3|\geq\theta\}|\geq s|\tubes_Y(Q)|^3.
\end{equation*}
Then
\begin{equation}\label{volumeBoundTrilinear}
\Big|\bigcup_{T\in\tubes}Y(T)\Big|\geq c_{\eps}\delta^{\eps} s^{9/4} \lambda_Y^{13/4}\delta^{3/4}\theta(\delta^3|\tubes|)^{1/4}.
\end{equation}
\end{cor}

\begin{proof}
After refining $(\tubes,Y)$, we can suppose that $|\tubes_Y(Q)|\sim\mu_Y$ for each $Q\in\mathcal{Q}_Y(T)$. Then
\begin{equation}
\begin{split}
&\delta^4|\mathcal{Q}(Y)|\big(\mu_Y^3\theta^{12/13}\big)^{13/27}\\
&\sim \sum_{Q\in\mathcal{Q}_Y(T)}|Q|\ \big(|\tubes_Y(Q)|^3\ \theta^{12/13}\big)^{13/27}\\
&\sim \sum_{Q\in\mathcal{Q}_Y(T)}|Q|\ s^{-1}\big(|\{T_1,T_2,T_3\in\tubes_Y(Q)\colon |v_1\wedge v_2\wedge v_3|\sim\theta\}|\ \theta^{12/13}\big)^{13/27}\\
&\leq s^{-1}\sum_{Q\in\mathcal{Q}_Y(T)}|Q| \big(\sum_{T_1,T_2,T_3\in\tubes_Y(Q)} |v_1\wedge v_2\wedge v_3|^{12/13}\big)^{13/27}\\
&\leq s^{-1}\int\Big(\sum_{T_1,T_2,T_3\in\tubes}\chi_{T_1}\ \chi_{T_2}\ \chi_{T_3}\ |v_1\wedge v_2\wedge v_3|^{12/13}\Big)^{13/27}\\
&\leq C_{\epsilon}s^{-1}\delta^{-1/3-\epsilon}(\delta^3|\tubes|)^{4/3}.
\end{split}
\end{equation}
Rearranging, we get
\begin{equation*}
\mu_Y^{13}\leq C_{\eps}^9 s^{-9} \delta^{-3-9\eps}\theta^{-4}(\delta^4|\mathcal{Q}(Y)|)^{-9}(\delta^3|\tubes|)^{12}.
\end{equation*}
Since $\mu_Y|\bigcup_{T\in\tubes}Y(T)|\approx \lambda_Y (\delta^3|\tubes|)$ and $\delta^4|\mathcal{Q}(Y)|\sim |\bigcup_{T\in\tubes}Y(T)|$, we have
\begin{equation*}
\lambda_Y^{-13}(\delta^3|\tubes|)^{13}|\bigcup_{T\in\tubes}Y(T)|^{-13}\leq C_{\eps}^9s^{-9}\delta^{-3-9\eps}\theta^{-4}|\bigcup_{T\in\tubes}Y(T)|^{-9}(\delta^3|\tubes|)^{12}.
\end{equation*}
Re-arranging, we obtain \eqref{volumeBoundTrilinear}.
\end{proof}

Equivalently, we get a pointwise bound
\begin{equation}\label{trilinearPointwise}
|\tubes_Y(Q)|\lessapprox \lambda^{-9/4}s^{-9/4}\theta^{-1}\delta^{-3/4}(\delta^3|\tubes|)^{3/4}\quad\textrm{for all}\ Q\in\mathcal{Q}(Y).
\end{equation}

Corollary \ref{trilinearVolumeBound} gives us a good bound if $\theta$ is large. If $\theta$ is small, then for a typical cube $Q\in\mathcal{Q}(Y)$, most of the tubes intersecting $Q$ will either all point in roughly the same direction, or they will be contained in the thin neighborhood of a plane. Section \ref{planebrushSection} will be devoted to handling this type of situation.

\section{Quantitative transversality}


\subsection{Concentration and non-concentration in planes}

\begin{defn}[Planyness]
Let $(\tubes,Y)$ be a set of $\delta$-tubes and their associated shading. We say that $(\tubes,Y)$ is plany if for each $\delta$-cube $Q$, there is a plane $\Pi(Q)$ so that for all $T\in\tubes_Y(Q)$ we have $\angle(v(T),\Pi(Q))\lesssim\delta$.
\end{defn}

\begin{defn}
Let $(\tubes,Y)$ be a set of plany $\delta$-tubes and their associated shading. Let $\delta\leq p \leq 1$. We say that $(\tubes,Y)$ is $(\eps_3,C_3)$-robustly contained in the $p$ neighborhood of planes if for each tube $T\in\tubes$, there exists a plane $\Pi(T)$ containing the line coaxial with $T$ so that 
\begin{itemize}
\item $\angle(\Pi(Q),\Pi(T))\leq p$ for each $Q\subset Y(T)$.
\item For every plane $\Pi$ containing the line coaxial with $T$ and for every $s>\delta$, we have
\begin{equation}\label{pConcentrationInPlanes}
|\{Q\subset Y(T)\colon \angle(\Pi(Q),\Pi)\leq s \}| \leq (s/p)^{\eps_3}C_3 \lambda_Y \delta^{-1}.
\end{equation}
\end{itemize}
\end{defn}

\begin{lem}\label{2ends2PlaneLem}
Let $(\tubes,Y)$ be a set of plany $\delta$-tubes and their associated shading. Let $\eps_3>0$. Then there is a number $\delta\leq p \leq 1$ and a $\lessapprox\delta^{\eps_3}$-refinement $(\tubes^\prime,Y^\prime)$ of $(\tubes,Y)$ that is $(\eps_3,100)$-robustly contained in the $p$ neighborhood of planes.
\end{lem}
\begin{proof}
For each $T\in\tubes$, define $p(T)$ to be the largest value of $s$ achieving the supremum below:
\begin{equation*}
 \sup_{\delta\leq s\leq 1} \sup_{\substack{\Pi\in\operatorname{Grass}(2;4)\\ T\ \textrm{coaxial with}\ \Pi}}
\frac{|\{Q\subset Y(T)\colon \angle(\Pi,\Pi(Q))\leq s\}|}{s^{\eps_3}}.
\end{equation*}
Setting $s=1$ we have 
\begin{equation*}
\sup_{\substack{\Pi\in\operatorname{Grass}(2;4)\\ T\ \textrm{coaxial with}\ \Pi}}
\frac{|\{Q\subset Y(T)\colon \angle(\Pi,\Pi(Q))\leq 1\}|}{1^{\eps_3}}=|\{Q\colon Q\subset Y(T)\}|,
\end{equation*}
and thus
\begin{equation*}
\sup_{\substack{\Pi\in\operatorname{Grass}(2;4)\\ T\ \textrm{coaxial with}\ \Pi}}
|\{Q\subset Y(T)\colon \angle(\Pi,\Pi(Q))\leq p(T) \}| \gtrsim p(T)^{\eps_3}|\{Q\colon Q\subset Y(T)\}|.
\end{equation*}

Let $\Pi(T)$ be a plane coaxial with $T$ that satisfies
\begin{equation*}
|\{Q\subset Y(T)\colon \angle(\Pi(T),\Pi(Q))\leq p(T)\}| \geq \frac{1}{2}p(T)^{\eps_3}|\{Q\colon Q\subset Y(T)\}|.
\end{equation*}
Define 
\begin{equation*}
Y_1(T)=\bigcup \{Q\subset Y(T)\colon \angle(\Pi(T),\Pi(Q))\leq p(T)\}.
\end{equation*}
Observe that $\sum_{T\in\tubes}|Y_1(T)|\gtrapprox \delta^{\eps_3}\sum_{T\in\tubes}|Y(T)|.$

After dyadic pigeonholing, we can select a number $p$ so that
\begin{equation*}
\sum_{\substack{T\in\tubes\\ p(T)\sim p}}|Y_1(T)|\gtrapprox \delta^{\eps_3}\sum_{T\in\tubes}|Y(T)|.
\end{equation*}
Finally, we can select a refinement $(\tubes^\prime,Y^\prime)$ of $(\tubes,Y_1)$ so that $|Y^\prime(T)|/|T|\sim\lambda_{Y^\prime}$ for each $T\in\tubes^\prime$. We thus have that $(\tubes^\prime,Y^\prime)$ is a $\approx\delta^{\eps_3}$-refinement of $(\tubes,Y)$; for each $T\in\tubes^\prime$ there is a plane $\Pi(T)$ so that $\angle(\Pi(Q),\Pi(T))\leq p$ for each $Q\subset Y^\prime(T)$; and for each plane $\Pi$ coaxial with $T$, we have the estimate
\begin{equation*}
|\{Q\subset Y^\prime(T)\colon \angle(\Pi(Q),\Pi)\leq s\}|\leq 100 (s/p)^{\eps_3}\lambda_{Y^\prime}\delta^{-1}.\qedhere
\end{equation*}
\end{proof}

If a set of tubes is robustly contained in planes, then this set of tubes breaks into non-interacting pieces.

\begin{lem}\label{containedInTwoPlanesDisjointness}
Let $(\tubes,Y)$ be a set of tubes that is $(\eps_3,100)$-robustly contained in the $p$ neighborhood of planes. Then there is a number $t\gtrsim 1$ and a $t$-refinement $(\tubes^\prime,Y)$ of $(\tubes,Y)$ so that $\tubes^\prime$ admits a partition
\begin{equation*}
\tubes^\prime=\bigsqcup_{i=1}^K \tubes_i
\end{equation*}
satisfying the following properties.
\begin{itemize}
\item For each index $i$, the tubes in $\tubes_i$ are contained in the $\lesssim p$ neighborhood of a plane $\Pi_i$. For each $T\in\tubes_i$, we have that $\angle(\Pi(T),\Pi_i)\leq p$.
\item If $T,T^\prime\in\tubes^\prime$ with $Y(T)\cap Y(T^\prime)\neq\emptyset$, then $T$ and $T^\prime$ are contained in the same set $\tubes_i$.
\end{itemize}
\end{lem}
\begin{proof}
By the triangle inequality, if $T,T^\prime\in\tubes$ with $Y(T)\cap Y(T^\prime)\neq\emptyset$, then
\begin{equation*}
\angle(\Pi(T),\Pi(T^\prime))\leq \angle(\Pi(T),\Pi(Q))+\angle(\Pi(T^\prime),\Pi(Q))\leq 2p,
\end{equation*}
where $Q$ is any cube contained in $Y(T)\cap Y(T^\prime)$. Let $\mathcal{G}$ be the set of all affine planes in $\RR^4$ that intersect the unit ball. Define 
\begin{equation*}
\operatorname{dist}(\Pi,\Pi^\prime)=\angle(\Pi,\Pi^\prime)+\operatorname{dist}\big(\Pi\cap B(0,1),\ \Pi^\prime\cap B(0,1)\big).
\end{equation*}
$\operatorname{dist}(\cdot,\cdot)$ defines a metric on $\mathcal{G}$. Note that if $Y(T)\cap Y(T^\prime) \neq \emptyset$, then $\operatorname{dist}(\Pi(T),\Pi(T^\prime))\leq 2p$. Consider the set of all balls of radius $p$ contained in $\mathcal{G}$. Select a set of balls $B_1,\ldots,B_K$ of radius $2p$ so that the 2-fold dilates of these balls are disjoint, and
\begin{equation*}
\sum_{i=1}^K \sum_{\substack{ T\in\tubes \\ \Pi(T)\in B_i }}|Y(T)|\gtrsim \sum_{T\in\tubes}|Y(T)|.
\end{equation*}
Define $\tubes_i=\{T\in\tubes\colon \Pi(T)\in B_i\}$. Since $\operatorname{dist}(\Pi(T),\Pi(T^\prime))>2p$ and thus $Y(T)\cap Y^\prime(T))=\emptyset$ whenever $T$ and $T^\prime$ come from different sets $\tubes_i$, we have that the sets $\{\tubes_i\}_{i=1}^K$ are disjoint. Define $\tubes^\prime=\bigcup_{i=1}^K\tubes_i$. We have that if $T,T^\prime\in\tubes^\prime$ with $Y(T)\cap Y(T^\prime)\neq\emptyset$, then $T$ and $T^\prime$ must be contained in the same set $\tubes_i$.
\end{proof}
Note that since $\angle(\Pi(T),\Pi_i)\leq p$ for each $T\in\tubes_i$, we might as well take $\Pi(T)=\Pi_i$; after we do this then each set $\tubes_i$ will be $(\eps_3,2^{1/\eps_3}100)$-robustly contained in the $2p$ neighborhood of planes. Combining the above two results, we obtain the following.
\begin{cor}\label{disjoint2PlanePieces}
Let $(\tubes,Y)$ be a set of plany $\delta$-tubes and their associated shading. Let $\eps_3>0$. Then there is a number $\delta \leq p \leq 1$ and a $\approx \delta^{\eps_3}$-refinement $(\tubes^\prime,Y^\prime)$ of $(\tubes,Y)$ so that $\tubes^\prime$ admits a partition  
\begin{equation*}
\tubes^\prime=\bigsqcup_{i=1}^K \tubes_i,
\end{equation*}
where
\begin{itemize}
\item For each index $i$ the tubes in $\tubes_i$ are $(\eps_3,C_3)$-robustly contained in the $p$ neighborhood of planes, where $C_3 = 2^{1/\eps_3}100$.
\item For each index $i$, the tubes in $\tubes_i$ are contained in the $\lesssim p$ neighborhood of a plane $\Pi_i$. For each $T\in\tubes_i$, we have $\Pi(T)=\Pi_i$.
\item If $T,T^\prime\in\tubes^\prime$ with $Y(T)\cap Y(T^\prime)\neq\emptyset$, then $T$ and $T^\prime$ are contained in the same set $\tubes_i$.
\end{itemize} 
\end{cor}

\subsection{Concentration and non-concentration in 3-planes}
\begin{defn}
Let $(\tubes,Y)$ be a set of essentially distinct plany $\delta$-tubes and their associated shading and let $\delta\leq\sigma\leq 1$. We say that $(\tubes,Y)$ is $(\eps_4,C_4)$-robustly contained in the $\sigma$ neighborhood of 3-planes if for each tube $T\in\tubes$, there exists a 3-plane $\Sigma(T)$ containing the line coaxial with $T$ so that 
\begin{itemize}
\item $\angle(\Pi(Q),\Sigma(T))\leq\sigma$ for each $Q\subset  Y(T)$.
\item For every 3-plane $\Sigma$ containing the line coaxial with $T$ and for every $s>0$, we have
\begin{equation*}
|\{Q \subset Y(T)\colon \angle(\Pi(Q),\Sigma)\leq s \}| \leq (s/\sigma)^{\eps_3}C_3 \lambda_Y\delta^{-1}.
\end{equation*}
\end{itemize}
\end{defn}

\begin{lem}\label{2ends3PlaneLem}
Let $(\tubes,Y)$ be a set of plany $\delta$-tubes and their associated shading and let $\eps_4>0$. Then there is a number $\delta\leq\sigma\leq 1$ and a $\sim \delta^{\eps_4}$-refinement $(\tubes^\prime,Y^\prime)$ of $(\tubes,Y)$ so that $|Y^\prime(T)|/|T|\sim\lambda_{Y^\prime}$ for all $T\in\tubes^\prime$, and $(\tubes,Y)$ is $(\eps_4,100)$-robustly contained in the $\sigma$ neighborhood of 3-planes.
\end{lem}
The proof of Lemma \ref{2ends3PlaneLem} is almost identical to the proof of Lemma \ref{2ends2PlaneLem}, so we will omit it.

\subsection{Uniqueness of $\Pi(T)$ and $\Sigma(T)$}
Suppose that $(\tubes,Y)$ is a set of plany $\delta$-tubes and their associated shading that is $(\eps_3,C_3)$ contained in the $p$ neighborhood of planes and $(\eps_4,C_4)$ contained in the $\sigma$ neighborhood of 3-planes, for some $\delta\leq\sigma\leq p\lesssim 1$. If $T\in\tubes$ and $Y(T)$ is non-empty, then the plane $\Pi(T)$ is unique (up to uncertainty $\sim p$) in the following sense: if $\Pi$ is a plane containing the line coaxial with $T$ and satisfying $\angle(\Pi(Q),\Pi)\leq p$ for at least one cube $Q\subset Y(T)$, then $\angle(\Pi,\Pi(T))\lesssim p$. The question of whether $\Sigma(T)$ is similarly unique is slightly more subtle. For example, if $p=\sigma$, then $\Sigma(T)$ is far from unique; every 3-plane containing $\Pi(T)$ would be an equally valid candidate for $\Sigma(T)$. While $\Sigma(T)$ might not be unique, the intersection of suitable neighborhoods of $\Pi(T)$ and $\Sigma(T)$ is unique. The next lemma and its corollary will make this statement precise.

\begin{lem}\label{prismsAreUnique}
Let $0<q\leq p\lesssim 1$. Let $\Pi\subset\RR^4$ be a plane, let $\Sigma,\Sigma^\prime\subset\RR^4$ be 3-planes containing $\Pi$. Let
\begin{equation*}
W=\{v\in S^3\colon \angle(v,\Pi)\leq p,\ \angle(v,\Sigma)\leq q, \angle(v,\Sigma^\prime)\leq q \}.
\end{equation*}
Then for each $1\leq A\leq p/(4q)$, at least one of the two following things must hold.
\begin{enumerate}
\item[(A):]\label{secondItem} There is a plane $\Pi_0\subset\RR^4$ containing the origin so that $W$ is contained in the $p/A$ neighborhood of $\Pi$.

\item[(B):]\label{firstItem} We have the containments
\begin{align}
\{v\in S^3\colon \angle(v,\Pi)\leq p,\ \angle(v,\Sigma)\leq q\} & \subset \{v\in S^3\colon \angle(v,\Pi)\leq p,\ \angle(v,\Sigma^\prime)\leq 6A q\},\label{firstContainment}\\
\{v\in S^3\colon \angle(v,\Pi)\leq p,\ \angle(v,\Sigma^\prime)\leq q\} & \subset \{v\in S^3\colon \angle(v,\Pi)\leq p,\ \angle(v,\Sigma)\leq 6A q\}\label{secondContainment}.
\end{align}

\end{enumerate}
\end{lem}
\begin{proof}
We will show that if Item (A) is false, then \eqref{secondContainment} is true. Since Item (A) is symmetric in $\Sigma$ and $\Sigma^\prime$, an identical argument shows that if Item (A) is false, then \eqref{firstContainment} is true,

Without loss of generality, we can assume that 
\begin{equation*}
\begin{split}
\Pi & = \{(x,y,z,w)\in\RR^4\colon z=w=0\},\\
\Sigma & =\{(x,y,z,w)\in\RR^4\colon w=0\},\\
\Sigma^\prime&= \{(x,y,z,w)\in\RR^4\colon az+bw=0\},
\end{split}
\end{equation*}
where $(a,b)$ is a unit vector. Then
\begin{equation*}
W = \{ (x,y,z,w)\in\RR^4: |(x,y,z,w)| = 1,  |z| < p, |w| < q, |az+bw|<q\}.
\end{equation*}
We will consider two cases.
\medskip

\noindent {\bf Case 1.} $a>2A q/p$. Then 
\begin{equation*}
\begin{split}
W & \subset \{ (x,y,z,w)\in\RR^4: |(x,y,z,w)| = 1,  |z| < p, |w| < q, |z|<2q/|a|\}\\
&\subset \{ (x,y,z,w)\in\RR^4: |(x,y,z,w)| = 1, |w| < q, |z|<p/A\},
\end{split}
\end{equation*}
so Item (A) holds. 

\medskip
\noindent {\bf Case 2.} $a\leq 2Aq/p$. Since $p/(4q)$, we have $a<1/2$ and thus $b\geq 1/2$. Thus
\begin{equation*}
\begin{split}
\{v\in S^3\colon &  \angle(v,\Pi)\leq p,\ \angle(v,\Sigma^\prime)\leq q\}\\
&=\{ (x,y,z,w)\in\RR^4: |(x,y,z,w)| = 1,  |z| < p, |w|<p, |az+bw|<q\}\\
&\subset \{ (x,y,z,w)\in\RR^4: |(x,y,z,w)| = 1,  |z| < p, |w|<2(2A+1)q\}.
\end{split}
\end{equation*}
Since $A\geq 1$, \eqref{secondContainment} holds. 
\end{proof}

\begin{cor}\label{prismCor}
Let $\delta\leq q\leq p\lesssim 1$ and let $(\tubes,Y)$ be a set of plany $\delta$-tubes and their associated shading that is $(\eps_3,C_3)$ contained in the $p$ neighborhood of planes. Let $T\in\tubes$ and let $t \geq C_3(p/8q)^{-\eps_3}$. Let $\Sigma,\Sigma^\prime$ be 3-planes containing the line coaxial with $T$ and suppose that
\begin{equation*}
|\{Q\subset Y(T)\colon \angle(\Pi(Q),\Sigma)\leq q,\ \angle(\Pi(Q),\Sigma^\prime)\leq q\}|\geq t \lambda_Y\delta^{-1}.
\end{equation*}
Then 
\begin{equation}\label{vectorContainments}
\begin{split}
\{v \in S^3\colon \angle(v,\Pi(T))\leq p,\ \angle(v,\Sigma)\leq q \} & \subset \{v \in S^3\colon \angle(v,\Pi(T))\leq 2p,\ \angle(v,\Sigma^\prime)\leq 13 (C_3/t)^{1/\eps_3}q \},\\
\{v \in S^3\colon \angle(v,\Pi(T))\leq p,\ \angle(v,\Sigma^\prime)\leq q \} & \subset \{v \in S^3\colon \angle(v,\Pi(T))\leq 2p,\ \angle(v,\Sigma)\leq 13 (C_3/t)^{1/\eps_3}q \}.
\end{split}
\end{equation}

\end{cor}
\begin{proof}
Without loss of generality, we can assume that the line coaxial with $T$ passes through the origin. Since $t\geq \delta/\lambda_Y,$ there is at least one cube $Q_0\subset Y(T)$ with $\angle(\Pi(Q_0),\Sigma)\leq q$ and $\angle(\Pi(Q_0),\Sigma^\prime)\leq q.$ Let $\Sigma_1$ (resp. $\Sigma_1^\prime$) be a 3-plane containing $\Pi(Q_0)$ with $\angle(\Sigma,\Sigma_1)\leq q$ (resp. $\angle(\Sigma,\Sigma_1^\prime)\leq q$). Note as well that since $(\tubes,Y)$ is $(\eps_3,C_3)$ contained in the $p$ neighborhood of planes, we have $\angle(\Pi(T),\Pi(Q_0))\leq p$. 

Apply Lemma \ref{prismsAreUnique} to $\Pi(Q_0),\Sigma_1,$ and $\Sigma_1^\prime$, with $p$ and $q$ as above and $A = 2(C_3/t)^{1/\eps_3}$ (the hypothesis $t \geq C_3(p/8q)^{-\eps_3}$ ensures that $A\leq p/(4q)$). We see that Item (A) cannot hold, since if it did, then there would exist a plane $\Pi_0$ so that
\begin{equation*}
\begin{split}
|\{Q\subset Y(T)\colon\angle(\Pi(Q),\Pi_0)\leq p/A\}|\geq & |\{Q\subset Y(T)\colon \angle(\Pi(Q),\Sigma)\leq q,\ \angle(\Pi(Q),\Sigma^\prime)\leq q\}|\\
&\geq t \lambda_Y\delta^{-1},
\end{split}
\end{equation*}
but this contradicts the estimate
\begin{equation*}
\begin{split}
|\{Q\subset Y(T)\colon\angle(\Pi(Q),\Pi_0)\leq p/A\}|&\leq C_3 A^{-\eps_3}\lambda_Y\delta^{-1}\\
&<t\lambda_y\delta^{-1}.
\end{split}
\end{equation*}
We conclude that Item (B) must hold, i.e.
\begin{equation}\label{vectorContainmentPre}
\begin{split}
\{v\in S^3\colon \angle(v,\Pi(Q_0)\leq p,\ \angle(v,\Sigma_1)\leq q\} & \subset \{v\in S^3\colon \angle(v,\Pi(Q_0))\leq p,\ \angle(v,\Sigma^\prime_1)\leq 12(C_3/t)^{1/\eps_3} q\},\\
\{v\in S^3\colon \angle(v,\Pi(Q_0))\leq p,\ \angle(v,\Sigma^\prime_1)\leq q\} & \subset \{v\in S^3\colon \angle(v,\Pi(Q_0))\leq p,\ \angle(v,\Sigma_1)\leq 12(C_3/t)^{1/\eps_3} q\}.
\end{split}
\end{equation}

Since $\angle(\Pi(Q_0),\Pi(Q))\leq p$, $\angle(\Sigma,\Sigma_1)\leq q$, and $\angle(\Sigma^\prime,\Sigma_1^\prime)\leq q$, \eqref{vectorContainmentPre} implies \eqref{vectorContainments}.
\end{proof}

\section{The Planebrush argument}\label{planebrushSection}
In this section, we will use the ``planebrush argument'' to show that unions of plany $\theta$-tubes must have large volume. The key geometric argument of the planebrush will be encapsulated in Lemma \ref{volumeBdPlainyTubesSpecialCase2Plane} below, which is quite technical. Lemma \ref{volumeBdPlainyTubesSpecialCase2Plane} is then used to prove Proposition \ref{volumeBdPlainyTubesSpecialCase}, which is a version of the lemma that removes some of the technical assumptions.

Up until this point, we have referred to $\delta$-tubes, $\delta$-cubes, etc. In this section, we will use the parameter $\theta$, and we will refer to $\theta$-tubes, $\theta$-cubes, etc. In later sections, we will simultaneously consider a Besicovitch set at two scales, $\delta$ and $\theta$, with $0<\delta<\!\!<\theta<\!\!<1$. The results from this section will be applied at scale $\theta$.

\begin{lem}\label{volumeBdPlainyTubesSpecialCase2Plane}
Let $0<\theta<1$. Let $0<\eps_3<\eps_2<\eps_1<1$. Let $\Omega\subset S^3$ be a set of $\theta$-separated directions. Let $(\tubes,Y)$ be a set of essentially distinct plany $\theta$-tubes and their associated shading.

Suppose that 
\begin{itemize}
\item There are $\sim|\tubes|/|\Omega|$ tubes from $\tubes$ pointing in each direction $v\in\Omega$. 
\item $(\tubes,Y)$ is $(\eps_1, C_1)$-two-ends.
\item $(\tubes,Y)$ is $(\eps_2,C_2)$-robustly-transverse.
\item There exists a plane whose $p$ neighborhood contains every tube from $\tubes$.
\item $(\tubes,Y)$ is $(\eps_3,C_3)$-robustly contained in the $p$ neighborhood of planes.
\end{itemize}
Then for each $\eps>0$, there is a constant $c>0$ (depending on  $\eps,\eps_1,\eps_2,$ and $\eps_3$) so that
\begin{equation}\label{volumeBoundPlainyTubesIn3Plane}
\Big|\bigcup_{T\in\tubes}Y(T)\Big|\geq c C_1^{-1/\eps_1} C_2^{-2/\eps_2} C_3^{-1/\eps_3}\theta^{\eps} \lambda_Y^{4/3} \theta^{2/3} (\theta^3|\Omega|)^{1/3}(\theta^3|\tubes|)^{2/3}.
\end{equation} 
\end{lem}

\begin{rem}
It might seem suspicious that the quantity $p$ appears in the hypothesis of Lemma \ref{volumeBdPlainyTubesSpecialCase2Plane} but does not appear in the conclusion. However, the quantity $p$ is implicitly present in \eqref{volumeBoundPlainyTubesIn3Plane}, since we always have the bound $|\Omega|\lesssim p^2 \theta^{-3}$.
\end{rem}
\begin{proof}
Fix a choice of $\eps>0$. Define $\eps_4=\eps/C_0,$ where $C_0$ is a large constant (depending on $\eps_1,\eps_2,$ and $\eps_3)$ that will be chosen later. Apply Lemma \ref{2ends3PlaneLem} to $(\tubes,Y)$ with this choice of $\eps_4$; let $\sigma$ and $(\tubes_1,Y_1)$ be the output from this lemma. We have that $(\tubes_1,Y_1)$ is a $\theta^{\eps_4}$-refinement of $(\tubes,Y)$, and $(\tubes_1,Y_1)$ is $(\eps_4,100)$-robustly contained in the $\sigma$ neighborhood of 3-planes. Furthermore, we have that 
\begin{equation}\label{sameNumberOfCubesInEachTube}
|\{Q\colon Q\subset Y_1(T)\}| \sim \theta^{-1}|Y_1(T)|/|T|\sim \lambda_{Y_1}\theta^{-1} \quad\textrm{for all}\ T\in\tubes_1.
\end{equation}

Note that
\begin{equation*}
\lambda_{Y_1}\geq\theta^{\eps_4}\lambda_Y,\quad\textrm{and}\quad\mu_{Y_1}\geq\theta^{\eps_4}\mu_Y.
\end{equation*}

Let $\mathcal{R}_1$ be the set of all quintuples $(Q,T,Q^\prime,T^\prime, Q^{\prime\prime})$ with
\begin{itemize}
\item $T,T^\prime\in\tubes_1$.
\item $Q,Q^\prime\subset Y_1(T)$.
\item $T^\prime\in\tubes_1(Q^\prime).$
\item $Q^{\prime\prime}\subset Y_1(T^\prime).$
\end{itemize}

We will estimate the size of $\mathcal{R}_1$. For each cube $Q^\prime\in\mathcal{Q}(Y_1)$, there are $\sim |\tubes_1(Q^\prime)|^2$ pairs of tubes $T,T^\prime\in\tubes_1(Q^\prime)$. By \eqref{sameNumberOfCubesInEachTube} we have that for each such pair, there are $\sim \theta^{-1}\lambda_{Y_1}$ cubes $Q\subset Y_1(T)$ and $\sim \theta^{-1}\lambda_{Y_1}$ cubes $Q^{\prime\prime}\subset Y_1(T^\prime)$. Thus for each $Q^\prime\in\mathcal{Q}(Y_1)$, there are $\sim |\tubes_1(Q^\prime)|^2 \theta^{-2}\lambda_{Y_1}^2$ quintuples $(Q,T,Q^\prime,T^\prime,Q^{\prime\prime})$ with that choice of $Q^\prime$. Summing over the cubes in $\mathcal{Q}(Y_1)$, we conclude that
\begin{equation}\label{upperBoundR_1}
|\mathcal{R}_1| \sim \theta^{-2}\lambda_{Y_1}^2 \sum_{Q^\prime\in\mathcal{Q}(Y_1)}|\tubes_1(Q^\prime)|^2\leq \theta^{-2}\lambda_{Y_1}^2|\mathcal{Q}(Y_1)| (\theta^{-\eps_4}C_2 \mu_{Y_1})^2,
\end{equation}
where in the final inequality we use the fact that $(\tubes_1,Y_1)$ is $(\eps_2,C_2\theta^{-\eps_4})$-robustly-transverse, so in particular $|\tubes_1(Q^\prime)|\leq \theta^{-\eps_4}C_2 \mu_{Y_1}$ for each cube $Q$.

Since $\sum_{Q\in\mathcal{Q}(Y_1)}|\tubes_1(Q)|= |\mathcal{Q}(Y_1)|\mu_{Y_1}$, by Cauchy-Schwarz we have
\begin{equation*}
\sum_{Q^\prime\in\mathcal{Q}(Y_1)}|\tubes_1(Q^\prime)|^2\geq |\mathcal{Q}(Y_1)|\mu_{Y_1}^2.
\end{equation*}
Thus
\begin{equation}\label{sizeOfR1}
|\mathcal{Q}(Y_1)|\theta^{-2}\lambda_{Y_1}^2\mu_{Y_1}^2 \leq |\mathcal{R}_1|\leq (\theta^{-\eps_4}C_2)^2|\mathcal{Q}(Y_1)|\theta^{-2}\lambda_{Y_1}^2\mu_{Y_1}^2.
\end{equation}

Define $\mathcal{R}_2$ to be the set of all quintuples $(Q,T,Q^\prime,T^\prime, Q^{\prime\prime})\in\mathcal{R}_1$ that satisfy the properties
\begin{align}
\operatorname{dist}(Q,Q^\prime)&\geq c_0 \theta^{\eps_4/\eps_1}C_1^{-1/\eps_1},\label{R2Prop1}\\
\operatorname{dist}(Q^\prime,Q^{\prime\prime})&\geq c_0 \theta^{\eps_4/\eps_1}C_1^{-1/\eps_1}\label{R2Prop4},\\
\angle(v(T),\ v(T^\prime))&\geq c_0\theta^{2\eps_4/\eps_2}C_2^{-2/\eps_2},\label{R2Prop2}\\
\angle(\Pi(Q),\Pi(Q^\prime))&\geq c_0 \theta^{\eps_4/\eps_3}C_3^{-1/\eps_3}p.\label{R2Prop3}
\end{align}

We claim that if the constant $c_0$ above is chosen sufficiently small, then 
\begin{equation}\label{R2VsR1}
|\mathcal{R}_2|\geq |\mathcal{R}_1|/2.
\end{equation}
Indeed, $\mathcal{R}_1\backslash\mathcal{R}_2$ is the set of quintuples $(Q,T,Q^\prime,T^\prime,Q^{\prime\prime})$ where at least one of the above four properties fails. We will show that the fraction of quintuples where any of these four properties fail is small.

For \eqref{R2Prop1},
\begin{equation}
\begin{split}
\sum_{Q^\prime\in\mathcal{Q}(Y_1)}&\sum_{T,T^\prime\in\tubes_1(Q^\prime)}\sum_{Q^{\prime\prime}\in Y_1(T^\prime)}|\{ Q\subset Y_1(T)\colon \operatorname{dist}(Q,Q^\prime)\leq c_0 C_1^{-1/\eps_1}\theta^{\eps_4/\eps_1} \}|\\
&\leq \sum_{Q^\prime\in\mathcal{Q}(Y_1)}\sum_{T,T^\prime\in\tubes_1(Q^\prime)}\sum_{Q^{\prime\prime}\in Y_1(T^\prime)}\big(c_0 C_1^{-1/\eps_1}\theta^{\eps_4/\eps_1}\big)^{\eps_1}C_1 \lambda_Y\theta^{-1}\\
&\leq \sum_{Q^\prime\in\mathcal{Q}(Y_1)}\sum_{T,T^\prime\in\tubes_1(Q^\prime)}\sum_{Q^{\prime\prime}\in Y_1(T^\prime)}c_0^{\eps_1} \lambda_{Y_1}\\
&\lesssim c_0^{\eps_1} \sum_{Q^\prime\in\mathcal{Q}(Y_1)} |\tubes_1(Q)|^2 \theta^{-2}\lambda_{Y_1}^2\\
&\sim c_0^{\eps_1}|\mathcal{R}_1|,
\end{split}
\end{equation}
where on the second-last line we used \eqref{sameNumberOfCubesInEachTube} and on the last line we used \eqref{sizeOfR1}. Thus if $c_0$ is selected sufficiently small (depending on $\eps_1$), then the set of quadruples $(Q,T,Q^\prime,T^\prime,Q^{\prime\prime})\in\mathcal{R}_1$ with $\operatorname{dist}(Q,Q^\prime)$ less than $c_0 C_1^{-1/\eps_1}\theta^{\eps_4/\eps_1}$ is less than $|\mathcal{R}_1|/8$.

An identical argument applies to \eqref{R2Prop4}. For \eqref{R2Prop2},
\begin{equation}
\begin{split}
\sum_{Q^\prime\in\mathcal{Q}(Y_1)}&\sum_{\substack{T,T^\prime\in \tubes_1(Q^\prime)\\ \angle(v(T),v(T^\prime))\leq c_0\theta^{2\eps_4/\eps_2}C_2^{-2/\eps_2} }}|\{(Q,Q^{\prime\prime}\colon (Q,T,Q^\prime,T^\prime,Q^{\prime\prime})\in\mathcal{R}_1 \}|\\
&\lesssim \sum_{Q^\prime\in\mathcal{Q}(Y_1)} |\tubes_1(Q^\prime)| \Big((c_0\theta^{2\eps_4/\eps_2}C_2^{-2/\eps_2})^{\eps_2}C_2\mu_Y\Big)(\lambda_{Y_1}^2\theta^{-2})\\
&\lesssim \sum_{Q^\prime\in\mathcal{Q}(Y_1)}\theta^{-2\eps_4}\mu_{Y_1}^2  (c_0\theta^{2\eps_4/\eps_2}C_2^{-2/\eps_2})^{\eps_2} C_2^2 \lambda_{Y_1}^2\theta^{-2}\\
&\lesssim c_0^{\eps_2}|\mathcal{R}_1|,
\end{split}
\end{equation}
Thus if $c_0$ is selected sufficiently small (depending on $\eps_2$), then the set of quadruples $(Q,T,Q^\prime,T^\prime,Q^{\prime\prime})\in\mathcal{R}_1$ with $\angle(v(T),\ v(T^\prime))$ less than $c_0\theta^{2\eps_4/\eps_2}C_2^{-2/\eps_2}$ is less than $|\mathcal{R}_1|/8$.

Finally, for \eqref{R2Prop3}, we have
\begin{equation}
\begin{split}
\sum_{Q^\prime\in\mathcal{Q}(Y_1)}&\sum_{T,T^\prime\in\tubes_1(Q^\prime)}\sum_{Q^{\prime\prime}\in Y_1(T^\prime)}|\{ Q\subset Y_1(T)\colon \angle(\Pi(Q),\Pi(Q^\prime))\leq c_0\theta^{\eps_4/\eps_3}C_3^{-1/\eps_3}p\}|\\
&\leq \sum_{Q^\prime\in\mathcal{Q}(Y_1)}\sum_{T,T^\prime\in\tubes_1(Q^\prime)}\sum_{Q^{\prime\prime}\in Y_1(T^\prime)}\big(c_0 \theta^{\eps_4/\eps_3}C_3^{-1/\eps_3}\big)^{\eps_3}C_3(\lambda\theta^{-1})\\
&\leq \sum_{Q^\prime\in\mathcal{Q}(Y_1)}\sum_{T,T^\prime\in\tubes_1(Q^\prime)}\sum_{Q^{\prime\prime}\in Y_1(T^\prime)}c_0^{\eps_3} \theta^{2\eps_3} C_3\theta^{-\eps_3}(\lambda_{Y_1}\theta^{-1})\\
&\lesssim c_0^{\eps_3}\sum_{Q^\prime\in\mathcal{Q}(Y_1)} |\tubes_1(Q)|^2 \theta^{-2}\lambda_{Y_1}^2\\
&\sim c_0^{\eps_3}|\mathcal{R}_1|.
\end{split}
\end{equation}
Thus if $c_0$ is selected sufficiently small (depending on $\eps_3$), then the set of quadruples $(Q,T,Q^\prime,T^\prime,Q^{\prime\prime})\in\mathcal{R}_1$ with $\angle(\Pi(Q),\Pi(Q^\prime))$ less than $c_0 \theta^{\eps_4/\eps_3}C_3^{-1/\eps_3}p$ is less than $|\mathcal{R}_1|/8$.

Thus if we select $c_0$ sufficiently small (depending only on $\eps_1,\eps_2,\eps_3,\eps_4$), then $|\mathcal{R}_2|\geq |\mathcal{R}_1|/2$, which establishes \eqref{R2VsR1}. For each quintuple $G=(Q,T,Q^\prime,T^\prime, Q^{\prime\prime})\in \mathcal{R}_2$, define

\begin{equation*}
\rho(G)=\max\{\angle(T^\prime, T_1^{\prime})\colon \exists\ T_1,Q_1^\prime,T_1^\prime\ \textrm{so that}\ (Q, T_1, Q_1^\prime, T_1^\prime,Q^{\prime\prime})\in\mathcal{R}_2\}.
\end{equation*}
See Figure \ref{defnRhoGFig}.
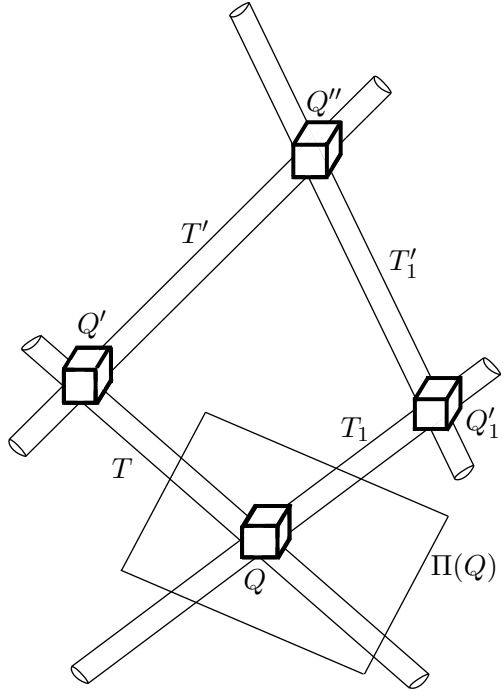
\begin{figure}[h!]
 \centering
\begin{overpic}[width=0.4\textwidth]{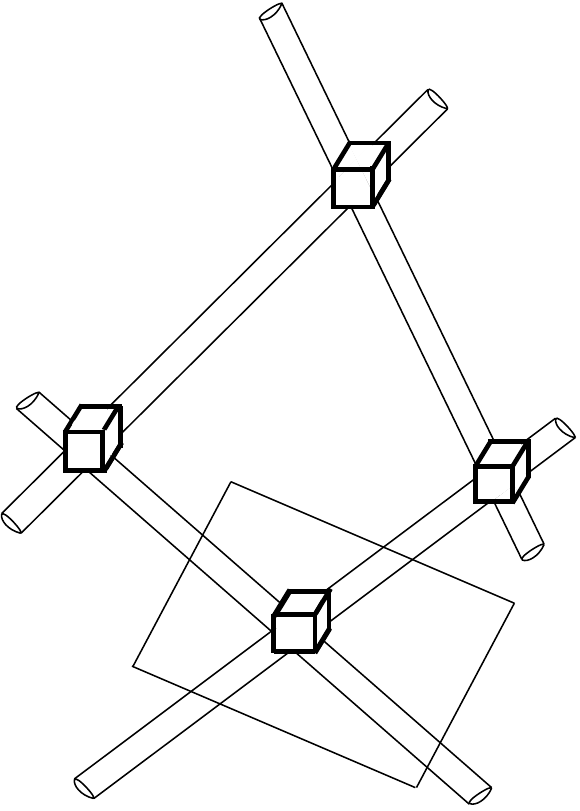}
 \put (34,15) {$Q$}
 \put (61,17) {$\Pi(Q)$}
 \put (15,31) {$T$}
 \put (10,52) {$Q^\prime$}
 \put (25,65) {$T^\prime$}
 \put (48,37) {$T_1$}
 \put (55,61) {$T_1^\prime$}
 \put (43,84) {$Q^{\prime\prime}$}
 \put (66,38) {$Q_1^{\prime}$}
\end{overpic}
 \caption{$\rho(G)$ is the angle between $T^\prime$ and $T_1^\prime$. In this figure, $T$ and $T_1$ make small angle with the plane $\Pi(Q)$ (so in particular, the cubes $Q^\prime$ and $Q_1^\prime$ are contained in the $\sim\theta$ neighborhood of $\Pi(Q)$. $T^\prime$ and $T_1^\prime$ do not make small angle with the plane $\Pi(Q)$.} \label{defnRhoGFig}
\end{figure}

A key geometric observation is that if $G=(Q,T,Q^\prime,T^\prime, Q^{\prime\prime})$, then $\rho(G)$ controls the angle between $\Pi(Q)$ and $\Sigma(T^\prime)$. Specifically, we have that if $\Sigma$ is the 3-plane containing $\Pi(Q)$ and the line coaxial with $T^\prime$, then
\begin{equation}\label{angleControl}
\angle\big(\Sigma,\ \Pi(Q^{\prime\prime})\big)\leq \frac{\theta^{1-\eps_4/\eps_1}}{C_1^{1/\eps_1}\rho(G)}.
\end{equation}
To see this, let $T_1^{\prime}$ be a tube maximizing $\rho(G)$ in the definition above. Let $L^\prime$ be a line with $|L^\prime\cap T^\prime|\sim 1$ so that $L^\prime$ intersects $Q^\prime$ and $\Pi(Q)$.  Let $L_1^{\prime}$ be a line with $|L_1^{\prime}\cap T_1^{\prime}|\sim 1$ so that $L_1^{\prime}$ intersects $L^\prime$; $L^{\prime}\cap L_1^{\prime}\in Q^\prime$, and $L_1^{\prime}$ intersects $\Pi(Q)$. Then $\Pi(Q^\prime)$ makes an angle $\lesssim \theta/\rho(G)$ with the plane spanned by $L^\prime$ and $L_1^{\prime}$. This plane is contained in the 3-plane spanned by $\Pi(Q)$ and $L_1^{\prime}$. Finally, since $\operatorname{dist}(Q,Q^\prime)\gtrsim \theta^{\eps_4/\eps_1}C_1^{-1/\eps_1}$, we have that $L^\prime$ makes an angle $\lesssim\theta^{1-\eps_4/\eps_1}C_1^{-1/\eps_1}$ with the line coaxial with $T^\prime$, and hence $\Pi(Q^\prime)$ makes an angle $\lesssim \frac{\theta^{1-\eps_4/\eps_1}}{C_1^{1/\eps_1}\rho(G)}$ with the 3-plane spanned by $\Pi(Q)$ and the line coaxial with $T^\prime$. See Figure \ref{linesAndPlaneFig}

\begin{figure}[h!]
 \centering
\begin{overpic}[width=0.3\textwidth]{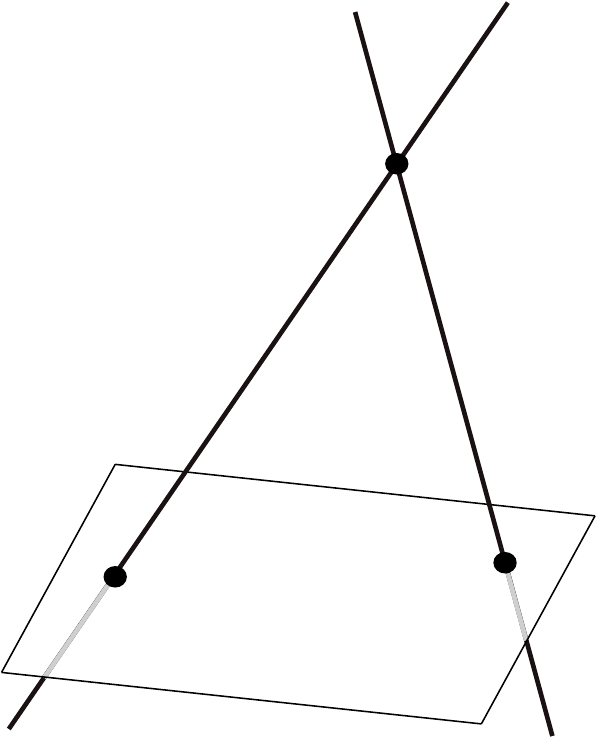}
\put (35,15) {$\Pi(Q)$}
\put (30,54) {$L^\prime$}
\put (62,54) {$L_1^\prime$}
\end{overpic}
 \caption{In this figure, black dots denote points of intersection. Since the three affine linear spaces $L^\prime$, $L_1^\prime$, and $\Pi(Q)$ all pairwise intersect, they must be contained in a common 3-plane. In particular, the plane spanned by $L^\prime$ and $L_1^\prime$ is contained in the 3-plane spanned by $\Pi(Q)$ and $L^\prime$.} \label{linesAndPlaneFig}
\end{figure}

After dyadic pigeonholing, we can select a value of $\rho$ so that there are $\geq |\mathcal{R}_2|/|\log\theta|$ quintuples $G\in\mathcal{R}_2$ with $\rho(G)\sim\rho$; call this set of quintuples $\mathcal{R}_3$. Next, we will show that
\begin{equation}\label{rhoBiggerSigma}
\rho\lesssim |\log\theta|^{1/\eps_4}\theta^{1-\eps_4/\eps_1}C_1^{1/\eps_1}\sigma^{-1}. 
\end{equation}

First, let $\mathcal{S}$ be the set of quadruples $(Q,T,Q^\prime,T^\prime)$ so that there exists a cube $Q^{\prime\prime}$ with $(Q,T,Q^\prime,T^\prime,Q^{\prime\prime})\in\mathcal{R}_2$. Note that for each $(Q,T,Q^\prime,T^\prime)\in\mathcal{S}$, there are at most $\theta^{-1}\lambda_{Y_1}$ cubes $Q^{\prime\prime}$ so that $(Q,T,Q^\prime,T^\prime,Q^{\prime\prime})\in\mathcal{R}_2$. On the other hand, an argument similar to the one used to establish the size of $\mathcal{R}_1$ shows that 
\begin{equation}\label{upperBoundonS}
|\mathcal{S}|\leq \sum_{Q\in\mathcal{Q}(Y_1)}|\tubes_1(Q)|^2 \theta^{-1}\lambda_{Y_1}^{-1}\sim |\mathcal{R}_1|(\theta\lambda_{Y_1}^{-1}).
\end{equation}

Define $\mathcal{S}_1$ to be the set of quadruples $(Q,T,Q^\prime,T^\prime)\in\mathcal{S}$ so that 

\begin{equation}\label{goodQuadruple}
|\{\mathcal{Q}^{\prime\prime}\colon(Q,T,Q^\prime,T^\prime,\mathcal{Q}^{\prime\prime}) \in\mathcal{R}_3\}|\geq c_1|\log\theta|^{-1}\theta^{-1}\lambda_{Y_1}.
\end{equation}
If the constant $c_1$ is chosen sufficiently small, then by \eqref{upperBoundonS} and the fact that $|\mathcal{R}_3|\gtrsim |\log\theta|^{-1}|\mathcal{R}_1|$, we have
\begin{equation*}
|\mathcal{S}_1|\gtrsim |\log\theta|^{-1}|\mathcal{S}|\gtrsim|\log\theta|^{-1}|\mathcal{R}_1|(\theta\lambda_{Y_1}^{-1}).
\end{equation*}

Fix a quadruple $(Q,T,Q^\prime,T^\prime)\in \mathcal{S}_1$. Observe that for each $\mathcal{Q}^{\prime\prime}\subset Y_1(T^\prime)$ with $(Q,T,Q^\prime,T^\prime,\mathcal{Q}^{\prime\prime}) \in\mathcal{R}_3$, by \eqref{angleControl} we have
\begin{equation*}
\angle\big(\Sigma,\ \Pi(Q^{\prime\prime})\big)\leq \frac{\theta^{1-\eps_4/\eps_1}}{C_1^{1/\eps_1}\rho(G)},
\end{equation*}
where $\rho(G)\sim\rho$ since $(Q,T,Q^\prime,T^\prime,\mathcal{Q}^{\prime\prime}) \in\mathcal{R}_3$. This implies
\begin{equation}\label{lowerBoundCubes}
\Big|\Big\{Q^{\prime\prime}\subset Y_1(T^\prime)\colon \angle\big(\Sigma,\ \Pi(Q^{\prime\prime})\big)\leq \frac{\theta^{1-\eps_4/\eps_1}}{C_1^{1/\eps_1}\rho}\Big\}\Big|\gtrsim |\log\theta|^{-1}\theta^{-1}\lambda_{Y_1}.
\end{equation}
On the other hand, since $(\tubes_1,Y_1)$ is $(\eps_4, 100)$-robustly contained in the $\sigma$ neighborhood of 3-planes, we have the estimate
\begin{equation}\label{upperBoundCubes}
\Big|\Big\{Q^{\prime\prime}\subset Y_1(T^\prime)\colon \angle\big(\Sigma,\ \Pi(Q^{\prime\prime})\big)\leq \frac{\theta^{1-\eps_4/\eps_1}}{C_1^{1/\eps_1}\rho}\Big\}\Big|\lesssim  \Big(\frac{\theta^{1-\eps_4/\eps_1}}{C_1^{1/\eps_1}\rho\sigma}\Big)^{\eps_4}\theta^{-1}\lambda_{Y_1}.
\end{equation}
Comparing \eqref{lowerBoundCubes} and \eqref{upperBoundCubes}, we conclude that 
\begin{equation*}
\Big(\frac{\theta^{1-\eps_4/\eps_1}}{C_1^{1/\eps_1}\rho\sigma}\Big)^{\eps_4}\gtrsim |\log\theta|^{-1},
\end{equation*}
and thus
\begin{equation*}
\theta/\rho\gtrsim |\log\theta|^{-1/\eps_4}\theta^{\eps_4/\eps_1}C_1^{1/\eps_1}\sigma.
\end{equation*}
Re-arranging, we obtain \eqref{rhoBiggerSigma}.

Next, by \eqref{angleControl}, each cube $Q^{\prime\prime}\in Y_1(T^\prime)$ with $(Q,T,Q^\prime,T^\prime,\mathcal{Q}^{\prime\prime}) \in\mathcal{R}_3$ satisfies
\begin{equation*}
\angle\big(\Pi(Q^{\prime\prime}), \Sigma\big)\leq \theta^{1-\eps_4/\eps_1}C_1^{1/\eps_1}\rho^{-1},
\end{equation*}
where $\Sigma$ is the 3-plane spanned by $\Pi(Q)$ and the line coaxial with $T^\prime$. Thus if $(Q,T,Q^\prime,T^\prime)\in\mathcal{S}_1$, then 
\begin{equation}\label{mostCubesCloseToSigma}
|\{Q^{\prime\prime}\in Y(T^\prime)\colon \angle\big(\Pi(Q^{\prime\prime}), \Sigma\big)\leq \theta^{1-\eps_4/\eps_1}C_1^{1/\eps_1}\rho^{-1}\}|\gtrsim |\log\theta|^{-1}\theta^{-1}\lambda_Y,
\end{equation}
and thus if we define $q=\theta^{1-\eps_4/\eps_1}C_1^{1/\eps_1}\rho^{-1}$ and $t=c|\log\theta|^{-1}$ for an appropriate constant $c\gtrsim 1$, then  
\begin{equation*}
|\{Q^{\prime\prime}\in Y(T^\prime)\colon \angle\big(\Pi(Q^{\prime\prime}, \Pi(T)\big)\leq p,\ \angle\big(\Pi(Q^{\prime\prime}), \Sigma(T)\big)\leq q,  \angle\big(\Pi(Q^{\prime\prime}), \Sigma\big)\leq q\}|\geq t \theta^{-1}\lambda_Y.
\end{equation*}
If we define 
\begin{equation}\label{defnOfQ}
q=C|\log\theta|^{1/\eps_4}\theta^{1-2\eps_4/\eps_3}C_1^{1/\eps_1}C_2^{2/\eps_2}C_3^{1/\eps_3} \rho^{-1}
\end{equation}
for an appropriate constant $C\lesssim 1$, then by Corollary \ref{prismCor} the sets of vectors
\begin{equation}
\{v\in S^3\colon \angle(v,\Pi(T^\prime)\leq p,\ \angle(v,\Sigma(T^\prime))\leq q\}
\end{equation}
and 
\begin{equation}\label{vectorsPiSigma}
\{v\in S^3\colon \angle(v,\Pi(T^\prime)\leq p,\ \angle(v,\Sigma)\leq q\}
\end{equation}
are comparable, in the sense that the $\sim 1$ dilate of the first set contains the second, and vice-versa. For each $T_0\in \tubes$, define  
\begin{equation}\label{defnWTp}
W(T_0)=\{v\in S^3\colon \angle(v,\Pi(T_0)\leq p,\ \angle(v,\Sigma(T_0))\leq q\}.
\end{equation}

We claim that if $(Q,T,Q^\prime,T^\prime)\in\mathcal{S}$, then the sets $W(T)$ and $W(T^\prime)$ are comparable, in the sense that the $\sim 1$ dilate of the first set contains the second, and vice-versa. We have already shown that $W(T^\prime)$ is comparable to the set \eqref{vectorsPiSigma}, so it suffices to show that $W(T)$ is comparable to the set \eqref{vectorsPiSigma} as well. By hypothesis, $\angle(\Pi(T_0),\Pi(T_1))\lesssim p$ for every pair of tubes $T_0,T_1\in\tubes$, so in particular $\angle(\Pi(T),\Pi(T^\prime))\lesssim p$. Next, note that $\Sigma(T)$ contains the planes $\Pi(Q)$ and $\Pi(Q^\prime)$. Since $T^\prime\in\tubes_1(Q^\prime)$, we have 
\begin{equation}\label{angleVTpPQp}
\angle(v(T^\prime),\Pi(Q^\prime))\leq \theta.
\end{equation}
We also have 
\begin{equation}\label{angleVtVtp}
\angle(v(T),v(T^\prime))\gtrsim\theta^{2\eps_4/\eps_2}C_2^{-2/\eps_2}
\end{equation}
Thus 
\begin{equation*}
\angle\big(\Pi(Q^\prime),\ \operatorname{span}(v(T),v(T))\big) \lesssim \theta^{1-2\eps_4/\eps_2}C_2^{2/\eps_2}.
\end{equation*}
Since 
\begin{equation*}
\angle(\Pi(Q),\Pi(Q^\prime))\gtrsim \theta^{\eps_4/\eps_3}C_3^{-1/\eps_3}p,
\end{equation*}
we conclude that 
\begin{equation*}
\angle(\Sigma,\Sigma(T))\lesssim \theta^{-2\eps_4/\eps_2 -\eps_4/\eps_3}C_2^{2/\eps_2}C_3^{1/\eps_3} \sigma/p,
\end{equation*}
where $\Sigma$ is the 3-plane spanned by $\Pi(Q)$ and the line coaxial with $T^\prime$. In particular, for any 
\begin{equation}\label{admissibilityOfQ}
q\geq\theta^{-2\eps_4/\eps_2 -\eps_4/\eps_3}C_2^{2/\eps_2}C_3^{1/\eps_3} \sigma,
\end{equation}
the sets of vectors
\begin{equation*}
\{v\in S^3\colon \angle(v,\Pi(T_0))\leq p,\ \angle(v,\Sigma(T))\leq q\}
\end{equation*}
and
\begin{equation*}
\{v\in S^3\colon \angle(v,\Pi(T_0))\leq p,\ \angle(v,\Sigma(T))\leq q\}
\end{equation*}
are comparable. By \eqref{rhoBiggerSigma}, the choice of $q$ given by \eqref{defnOfQ} satisfies \eqref{admissibilityOfQ}, so in particular $W(T)$ is comparable to the set \eqref{vectorsPiSigma}, and thus is comparable to $W(T^\prime)$.  

Next we will count the number of triples $(T,Q^\prime,T^\prime)$ so that $(Q,T,Q^\prime,T^\prime)\in\mathcal{S}_1$ for at least one cube $Q$; denote this set of triples by $\mathcal{T}$. If $(T,Q^\prime,T^\prime)$ is such a triple, then there are $\leq \theta^{-1}\lambda_{Y_1}$ cubes $Q$ so that $(Q,T,Q^\prime,T^\prime)\in\mathcal{S}_1$. This implies that
\begin{equation}
|\mathcal{T}|\geq \theta\lambda_{Y_1}^{-1}|\mathcal{S}_1|\gtrsim |\log\theta|^{-1}\sum_{Q\in\mathcal{Q}(Y_1)}|\tubes_1(Q)|^2.
\end{equation}

Let $\mathcal{W}$ be a set of essentially distinct subsets of $S^3$ of the form 
\begin{equation*}
\{v\in S^3\colon \angle(v,\Pi_0)\leq 2p,\ \angle(v,\Sigma)\leq 2q\},
\end{equation*}
where $\Pi_0$ is the plane whose $p$-neighborhood contains every tube from $\tubes$, and $\Sigma$ is a 3-plane in $\RR^4$ containing the origin. For each $W\in\mathcal{W}$, define 
\begin{equation*}
\tubes_W=\{T\in\tubes_1\colon W(T) \subset W\},
\end{equation*}
where $W(T)$ is defined in \eqref{defnWTp}. Observe that 
\begin{equation*}
\sum_{W\in\mathcal{W}}|\tubes_W|\sim \tubes_1,
\end{equation*}
and if $(T,Q^\prime,T^\prime)\in\mathcal{T},$ then there exists some $W\in\mathcal{W}$ so that $T\in\tubes_W$ and $T^\prime\in\tubes_W$. This implies
\begin{equation*}
\sum_{W\in\mathcal{W}}\sum_{Q\in\mathcal{Q}(Y_1)} |\tubes_W(Q)|^2 \gtrsim |\log\theta|^{-1} \sum_{Q\in\mathcal{Q}(Y_1)}|\tubes_1(Q)|^2,
\end{equation*}
and thus by Cauchy-Schwarz, 
\begin{equation}\label{breaksIntoDisjointBalls}
\Big|\bigcup_{T\in\tubes_1}Y_1(T)\Big|\geq |\log\theta|^{-1}\sum_{W\in\mathcal{W}}\Big|\bigcup_{T\in\tubes_W} Y_1(T)\Big|.
\end{equation}

Thus there exists a $\sim |\log\theta|^{-1}$-refinement $Y_2$ of $Y_1$ so that the sets
\begin{equation*}
\Big\{ \bigcup_{T\in\tubes_W}Y_2(T) \Big\}_{W\in\mathcal{W}}
\end{equation*}
are disjoint. This means that each cube $Q\in\mathcal{Q}(Y_2)$ can be uniquely associated to a set $W(Q)\in\mathcal{W}$. Refining the set $\mathcal{W}$ slightly, we can assume that no set $W(T)$ is contained in more than one set $W\in\mathcal{W}$. This refinement of $\mathcal{W}$ induces a $\gtrsim 1$-refinement of $(\tubes_1,Y_1)$; call this new set $(\tubes_2,Y_2)$. Thus each tube $T\in\tubes_2$ can also be uniquely associated to a set in $\mathcal{W}$. Abusing notation slightly, we will call this set $W(T)$ (so now $W(T)$ is always an element of $\mathcal{W}$).

Next, define $\mathcal{R}_4$ to be the set of all quintuples $(Q,T,Q^\prime,T^\prime,Q^{\prime\prime})\in\mathcal{R}_3$ so that $T,T^\prime\in \tubes_2$; $Q,Q^\prime\subset Y_2(T);$ $Q^\prime,Q^{\prime\prime}\subset Y_2(T^\prime)$; if we choose the refinement of $\mathcal{W}$ appropriately, then 
\begin{equation*}
|\mathcal{R}_4|\gtrsim  |\log\theta|^{-O(1)}|\mathcal{R}_3|.
\end{equation*}

Observe that if $(Q,T,Q^\prime,T^\prime,Q^{\prime\prime})\in\mathcal{R}_4$, then $T,Q^\prime,T^\prime,$ and $Q^{\prime\prime}$ are all associated to the same set $W\in\mathcal{W}$. Note that 

\begin{equation*}
|\mathcal{R}_4|\gtrsim |\log\theta|^{-O(1)}|\mathcal{R}_1|\gtrsim |\log\theta|^{-O(1)}|\mathcal{Q}_{Y_1}|\theta^{-2}\lambda_{Y_1}^2\mu_{Y_1}^2.
\end{equation*}
Thus if we define 
\begin{equation*}
\mathcal{R}_5=\{(Q,T,Q^\prime,T^\prime,Q^{\prime\prime})\in \mathcal{Q}_4\colon Q\ \textrm{is part of}\ \geq c|\log\theta|^{-C}\theta^{-2}\lambda_{Y_1}^2\mu_{Y_1}^2 \textrm{quintuples from}\ \mathcal{Q}_4\},
\end{equation*}
then if the constant $c$ is selected sufficiently small and $C$ is selected sufficiently large, we have $|\mathcal{R}_5|\geq|\mathcal{R}_4|/2$. 

By dyadic pigeonholing, we can select a set $\mathcal{W}^\prime\subset\mathcal{W}$ so that each set $W\in\mathcal{W}^\prime$ has roughly the same number of quintuples from $\mathcal{R}_5$ associated to it; define $\mathcal{R}_6$ to be the set of quintuples contained in a set from $\mathcal{W}^\prime$. Define $(\tubes_3,Y_3)$ to be the refinement of $(\tubes_2,Y_2)$ consisting of tubes and cubes associated to some set from $\mathcal{W}^\prime$. Since $|\mathcal{R}_6|\gtrsim |\log\theta|^{-O(1)}|\mathcal{R}_1|$, we have that $(\tubes_3,Y_3)$ is a $\gtrsim |\log\theta|^{-O(1)}$-refinement of $(\tubes_2,Y_2)$. 

Thus if $W_0\in\mathcal{W}$ is the set that minimizes $\big|\bigcup_{T\in\tubes_{W_0}\cap\tubes_3}Y_3(T)\big|$, then 
\begin{equation}
\Big|\bigcup_{T\in\tubes_1}Y_1(T)\Big|\gtrsim |\log\theta|^{-O(1)}|\mathcal{W}|\Big|\bigcup_{T\in\tubes_3\cap\tubes_{W_0}}Y_3(T)\Big|.
\end{equation}

We have $|\tubes_3|\gtrsim |\log\theta|^{-O(1)}|\tubes|,$ and by the hypotheses of Lemma \ref{volumeBdPlainyTubesSpecialCase2Plane}, we have that $\sim |\tubes|/|\Omega|$ tubes from $\tubes$ point in each direction $v\in\Omega$. Thus 
\begin{equation}\label{numberOfDirectionsT3}
|\{v(T)\colon T\in\tubes_3\}|\gtrsim|\log\theta|^{-O(1)}|\Omega|.
\end{equation}

Note that for each $W\in\mathcal{W}$, the set of directions of tubes in $\tubes_3\cap\tubes_{W}$ is contained in a subset of $S^3$ of dimensions $\times 1\times p\times q$. This means that the set of directions of tubes in $\tubes_3\cap\tubes_W$ is contained in a rectangle in $S^3$ of dimensions $\sim 1\times p \times q$. Since the possible directions of tubes in $\tubes_3\cap\tubes_W$ are $\theta$-separated, we have that
\begin{equation}\label{numberOfDirectionsTB}
|\{v(T)\colon T\in \tubes_3\cap\tubes_{W}\}| \lesssim pq \theta^{-3}\lessapprox |\log\theta|^{1/\eps_4}\theta^{-2\eps_4/\eps_3}C_1^{1/\eps_1}C_2^{2/\eps_2}C_3^{1/\eps_3}\rho^{-1}p\theta^{-2}.
\end{equation}

Thus

\begin{equation}
\begin{split}
|\mathcal{W}|&\gtrsim |\log\theta|^{-O(1)} \frac{|\Omega|}{\theta^{-2\eps_4/\eps_3}C_1^{1/\eps_1}C_2^{2/\eps_2}C_3^{1/\eps_3}\rho^{-1}p\theta^{-2}},
\end{split}
\end{equation}
where the implicit constant may depend on $\eps_1,\ldots,\eps_4$. We conclude that 
\begin{equation}\label{allTubesBoundedByPart}
\Big|\bigcup_{T\in\tubes_1}Y_1(T)\Big|\gtrsim  |\log\theta|^{-O(1)} \theta^{2\eps_4/\eps_3}C_1^{-1/\eps_1}C_2^{-2/\eps_2}C_3^{-1/\eps_3} (\theta^3|\Omega|)(\theta^{-1}\rho p^{-1})\Big|\bigcup_{T\in\tubes_3\cap\tubes_{W_0}}Y_3(T)\Big|.
\end{equation}

It remains to estimate $\Big|\bigcup_{T\in\tubes_3\cap\tubes_{W_0}}Y_3(T)\Big|$. Let $(Q_0,T_0,Q_0^\prime,T_0^\prime,Q_0^{\prime\prime})$ be a quintuple associated to $W_0$. Since $(Q_0,T_0,Q_0^\prime,T_0^\prime,Q_0^{\prime\prime})\in\mathcal{R}_5$, there are $\gtrapprox \theta^{-2}\lambda_{Y_1}^2\mu_{Y_1}^2$ quadruples $(T,Q^\prime,T^\prime,Q^{\prime\prime})$ so that the quintuple $(Q_0,T,Q^\prime,T^\prime,Q^{\prime\prime})$ is an element of $\mathcal{R}_4$ and is associated to $W_0$.

Next, we will estimate: amongst these $\gtrapprox \theta^{-2}\lambda_{Y_1}^2\mu_{Y_1}^2$ quadruples, how many distinct cubes $Q^{\prime\prime}$ occur? This quantity is relevant since the volume of $\Big|\bigcup_{T\in\tubes_3\cap\tubes_{W_0}}Y_3(T)\Big|$ is at least $\theta^4$ times the number of distinct cubes $Q^{\prime\prime}$.

\begin{itemize}
\item For each cube $Q^{\prime\prime}$, the set of potential choices of $T^\prime$ must all point in directions that make angle $\lesssim\theta$ with the plane $\Pi(Q^{\prime\prime})$, and these directions must lie in an angular sector of aperture $\lesssim\rho$; thus there are 
\begin{equation}\label{numberTPrime}
\lesssim \theta^{-1}\rho
\end{equation} 
choices for $T^{\prime}$. 
\item For each tube $T^\prime$, the set of potential choices of $Q^\prime$ must all lie in $T^\prime\cap N_{\theta}(\Pi(Q_0))$. By \eqref{R2Prop2} we have $\angle(v(T),v(T^\prime))\gtrsim \theta^{2\eps_4/\eps_2}C_2^{-2/\eps_2}$, and by \eqref{R2Prop3} we have $\angle(\Pi(Q_0),\Pi(Q))\gtrsim \theta^{\eps_4/\eps_3}C_3^{-1/\eps_3}p$; thus $\angle(v(T^\prime),\Pi(Q_0))\gtrsim \theta^{2\eps_4/\eps_2+\eps_4/\eps_3}C_2^{-2/\eps_2}C_3^{-1/\eps_3}p$. We conclude that for each $T^\prime$, there are 
\begin{equation}\label{numberQPrime}
\lesssim \theta^{-2\eps_4/\eps_2-\eps_4/\eps_3}C_2^{2/\eps_2}C_3^{1/\eps_3}p^{-1}
\end{equation}
choices for $Q^\prime$. 
\item For each cube $Q^\prime$, the set of potential $T$ must intersect both $Q_0$ and $Q^\prime$. By \eqref{R2Prop1}, $\operatorname{dist}(Q_0,Q^\prime)\gtrsim \theta^{\eps_4/\eps_1}C_1^{1/\eps_1}$. Since the tubes point in $\theta$-separated directions and satisfy $\angle(v(T),\Pi(Q_0))\leq\theta$, there are 
\begin{equation}\label{numberT}
\lesssim\theta^{-\eps_4/\eps_1}C_1^{1/\eps_1}
\end{equation}
choices for $T$.
\end{itemize}

Multiplying the bounds in \eqref{numberTPrime}, \eqref{numberQPrime}, and \eqref{numberT} and using the fact that $\eps_3<\eps_2<\eps_1$, we conclude that each cube $Q^{\prime\prime}$ is part of $\lesssim \theta^{-4\eps_4/\eps_3}C_1^{1/\eps_1} C_2^{2/\eps_2}C_3^{1/\eps_3}p^{-1}\theta^{-1}\rho$ quintuples. Thus we have the volume bound 
\begin{equation}
\begin{split}
\Big|\bigcup_{T\in\tubes_3\cap\tubes_{W_0}}Y_3(T)\Big|&\gtrapprox \theta^4  \big(\theta^{4\eps_4/\eps_3}C_1^{-1/\eps_1} C_2^{-2/\eps_2}C_3^{-1/\eps_3}p\theta\rho^{-1}\big) \mu_{Y_1}^2\lambda_{Y_1}^2\theta^{-2}.
\end{split}
\end{equation}

Inserting this bound into \eqref{allTubesBoundedByPart}, we have 
\begin{equation}
\begin{split}
\Big|\bigcup_{T\in\tubes_1}Y_1(T)\Big| & \gtrsim |\log\theta|^{-O(1)}\theta^{6\eps_4/\eps_3}C_1^{-2/\eps_1} C_2^{-4/\eps_2}C_3^{-2/\eps_3}(\theta^3|\Omega|)\mu_{Y_1}^2\lambda_{Y_1}^2\theta^2.
\end{split}
\end{equation}

Now, we have $\mu_{Y_1}\Big|\bigcup_{T\in\tubes_1}Y_1(T)\Big| \sim \lambda_{Y_1} (\theta^3|\tubes_1|)$, or 
\begin{equation*}
\mu_{Y_1}\sim \frac{\lambda_{Y_1} (\theta^3|\tubes_1|)}{\Big|\bigcup_{T\in\tubes_1}Y_1(T)\Big|}.
\end{equation*}
Re-arranging, and recalling that $\lambda_{Y_1}\gtrsim\theta^{\eps_4}\lambda_Y$ we conclude that there exists a constant $C_0$ (depending on $\eps_1,\eps_2,\eps_3$ so that
\begin{equation}\label{finalUnionBound}
\begin{split}
\Big|\bigcup_{T\in\tubes}Y(T)\Big|^3&\geq \Big|\bigcup_{T\in\tubes_1}Y_1(T)\Big|^3 \\
&\gtrsim  |\log\theta|^{-C_0} \theta^{C_0\eps_4} C_1^{-2/\eps_1}C_2^{-4/\eps_2}C_3^{-2/\eps_3}\lambda_{Y_1}^4\theta^{2}(\theta^3|\Omega|) (\theta^3|\tubes|)^2.
\end{split}
\end{equation}
If $\eps_4$ is selected sufficiently small (depending on $\eps$ and $C_0$, which in turn depends on $\eps_1,\eps_2,\eps_3)$, then \eqref{finalUnionBound} implies \eqref{volumeBoundPlainyTubesIn3Plane}
\end{proof}

The next result will remove the requirement that the tubes be contained in the $p$ neighborhood of planes.

\begin{prop}\label{volumeBdPlainyTubesSpecialCase}
Let $0<\theta<1$. Let $0<\eps_2<\eps_1<1$. Let $\Omega\subset S^3$ be a set of $\theta$-separated directions. Let $(\tubes,Y)$ be a set of essentially distinct plany $\theta$-tubes and their associated shading. 

Suppose that 

\begin{itemize}
\item There are $\sim|\tubes|/|\Omega|$ tubes from $\tubes$ pointing in each direction $v\in\Omega$. 
\item $(\tubes,Y)$ is $(\eps_1,C_1)$-two-ends.
\item $(\tubes,Y)$ is $(\eps_2,C_2)$-robustly-transverse.
\end{itemize}
Then for each $\eps>0$, there is a constant $c>0$ (depending on $\eps$, $\eps,\eps_1,$ and $\eps_2$) so that
\begin{equation}\label{volumeBoundPlainyTubesPropEqn}
\Big|\bigcup_{T\in\tubes}Y(T)\Big|\geq c C_1^{-1/\eps_1}C_2^{-2/\eps_2}\theta^{\eps}\lambda_Y^{4/3} \theta^{2/3} \big(\theta^3|\Omega|\big)^{1/3}\big(\theta^3|\tubes|\big)^{2/3}.
\end{equation} 
\end{prop}

\begin{proof}
Let $\eps>0$. Define $\eps_3=c_3\eps\eps_2$, where $c_3>0$ is a constant that will be determined below. Apply Corollary \ref{disjoint2PlanePieces} to $(\tubes,Y)$ with parameter $\eps_3$, and let $\theta<p<1$ and $(\tubes^{\prime},Y^{\prime})$ be the output from the Corollary. We have that $(\tubes^{\prime},Y^\prime)$ is a $\theta^{\eps_3}$-refinement of $(\tubes,Y)$, and $\tubes^{\prime}=\bigsqcup_{i=1}^K\tubes_i$, where
\begin{equation*}
\Big|\bigcup_{T\in\tubes^\prime}Y^\prime(T)\Big|\sim\sum_{i=1}^K\Big|\bigcup_{T\in\tubes_i}Y^\prime(T)\Big|
\end{equation*}
and each set $\tubes_i$ is contained in the $\lesssim p$-neighborhood of a $2$-plane and is $(\eps_3,100)$-robustly contained in the $p$ neighborhood of planes.

By pigeonholing, there exists a set of indices $I\subset\{1,\ldots,K\}$ so that for each $j\in I$, we have
\begin{equation*}
\sum_{T\in\tubes_j}|Y(T)|\sim \frac{1}{|I|} \sum_{i\in I}\sum_{T\in\tubes_i}|Y(T)|\geq|\log\theta|^{-1}  \sum_{T\in\tubes}|Y(T)|.
\end{equation*}
For each index $i\in I$, define the shading $Y_i(T)=Y(T)$ for each $T\in\tubes_i$. Note that $\lambda_{Y_i}\geq |\log\theta|^{-1}\theta^{\eps_3}\lambda_Y$ for each $i\in I$. We have that for each $i\in I$, $(\tubes_i,Y_i)$ is $(\eps_1,|\log\theta|\theta^{-\eps_3}C_1)$-two-ends, $(\eps_2,|\log\theta|\theta^{-\eps_3}C_2)$-robustly-transverse, and $(\eps_3,100|\log\theta|)$-robustly contained in the $p$ neighborhood of planes.

For each index $i\in I$, let $(\tubes_i^\prime,Y_i^\prime)$ be a refinement of $(\tubes_i,Y_i)$ and let $\Omega_i^\prime\subset\Omega$ so that there are $\sim|\tubes_i^\prime|/|\Omega_i^\prime|$ tubes from $\tubes_i^\prime$ pointing in each direction $v\in\Omega_i^\prime.$ We have that $(\tubes_i^\prime,Y_i^\prime)$ is $(\eps_1,|\log\theta|^2\theta^{-\eps_3}C_1)$-two-ends and $(\eps_2,|\log\theta|^2\theta^{-\eps_3}C_2)$-robustly-transverse, and $(\eps_3,100|\log\theta|^2)$-robustly contained in the $p$ neighborhood of planes.

Thus for each index $i\in I$, the pair $(\tubes_i^\prime,Y^{\prime})$ satisfies the hypotheses of Lemma \ref{volumeBdPlainyTubesSpecialCase2Plane}. Applying Lemma \ref{volumeBdPlainyTubesSpecialCase2Plane} with $\eps/2$ in place of $\eps$, we conclude that for each index $i\in I$, 
\begin{equation*}
\Big|\bigcup_{T\in\tubes_i^\prime}Y^{\prime\prime\prime}(T)\Big|\geq W\lambda_Y^{4/3}\theta^{2/3}(\theta^3|\Omega_i^\prime|)^{1/3}(\theta^3|\tubes_i^\prime|)^{2/3},
\end{equation*}
where 
\begin{equation}\label{defnW}
W = c^\prime|\log\theta|^{-C}C_1^{-1/\eps_1}C_2^{-2/\eps_2} \theta^{(\eps_3/\eps_1+\eps_3/\eps_2)} \theta^{\eps/2}.
\end{equation}

Summing over all indices $i$ and applying H\"older's inequality, we conclude
\begin{equation}\label{volumeBdBigcupYT}
\begin{split}
\Big|\bigcup_{T\in\tubes}Y(T)\Big|&\geq \sum_{i\in I}\Big|\bigcup_{T\in\tubes_i^\prime}Y^{\prime}(T)\Big|\\
&\geq W\lambda_Y^{4/3}\theta^{2/3}\sum_{i\in I} (\theta^3|\Omega_i^\prime|)^{1/3}(\theta^3|\tubes_i^\prime|)^{2/3}\\
&\geq W\lambda_Y^{4/3}\theta^{2/3}\Big(\theta^3\sum_{i\in I}|\Omega_i^\prime|\Big)^{1/3}\Big(\theta^3\sum_{i\in I}|\tubes_i^\prime|\Big)^{2/3}.
\end{split}
\end{equation}
We have $\Big|\bigcup_{i\in I}\Omega_i\Big|\gtrapprox |\Omega|$ and $\sum_{i\in I}|\tubes_i^\prime|\gtrapprox \theta^{\eps_2}|\tubes|$. Thus
\begin{equation}\label{unionOfOmegaAndTubes}
\Big|\bigcup_{T\in\tubes}Y(T)\Big| \gtrapprox W\lambda_Y^{4/3}\theta^{2/3}(\theta^3|\Omega|)^{1/3}(\theta^3|\tubes|)^{2/3}.
\end{equation}
Thus if $c_3>0$ and $c>0$ are selected sufficiently small, then \eqref{volumeBoundPlainyTubesPropEqn} holds.
\end{proof}

\section{Volume bounds for unions of weakly plany tubes}\label{volumeBoundsPlainySec}

In this section we will use random sampling and re-scaling arguments to weaken two of the hypotheses of Proposition \ref{volumeBdPlainyTubesSpecialCase}. First, we will remove the requirement that the collection of tubes be robustly transverse. Second, we will replace the requirement that that all of the tubes passing through a point lie in the $\theta$ neighborhood of a plane with the weaker requirement that these tubes be contained in the $\theta$ neighborhood of a union of planes\footnote{Of course, as the number of planes in the union increases, our bounds will become weaker.}. This result will be stated precisely in Proposition \ref{volumeBdPlainyTubes} below. 

Before we begin, it will be useful to see how the planebrush argument fits into the broader proof strategy for proving Kakeya estimates. In short, given a collection $\tubes$ of $\delta$-tubes in $\RR^4$, we can find a parameter $\delta\leq\theta\leq 1$ so that for each cube $Q\in\mathcal{Q}(Y)$, the tubes in $\tubes_Y(Q)$ make angle at most $\theta$ with a plane, and $\theta$ is the smallest number for which this property holds. We then examine the set $\tubes$ at scale $\theta$ and discover that the tubes from $\tubes$ intersecting a typical $\theta$-cube also cluster into planes. We will then apply the planebrush argument at scale $\theta$.

The following rather technical lemma makes the above statement precise. 

\begin{lem}[Tubes are either plany or trilinear]\label{planyOrTrilinearLem}
Let $(\tubes,Y)$ be a set of $\delta$-tubes and their associated shading. Suppose that $(\tubes,Y)$ is $(\epsilon_1,C_1)$-two-ends and $(\epsilon_2,C_2)$-robustly transverse.

Then there exists a number $\delta\leq\theta\leq 1$; a refinement $(\tubes^\prime,Y^\prime)$ of $(\tubes,Y)$; a set $\Omega\subset S^3$ of $\theta$-separated points; a set $(\tubes_{\theta},Y_\theta)$ of essentially distinct $\theta$-tubes and their associated shading; and numbers $B,B_1,B_2\geq 1$ with $B_1\geq B_2$ and $B_1/B_2 \leq B$ so that the following holds.
\begin{enumerate}[label=(\alph*)]
\item\label{deltaCubePlane} For each $Q\in\mathcal{Q}(Y^\prime)$, there is a plane $\Pi(Q)$ so that $\angle(v(T),\Pi(Q))\leq\theta$ for all $T\in\tubes_Y(Q)$.

\item\label{mostTriplesWedgeThetaItem} For each $Q\in\mathcal{Q}(Y^\prime)$,
\begin{equation}\label{mostTriplesWedgeTheta}
|\{T_1,T_2,T_3\in \tubes_{Y^\prime}(Q)\colon |v(T_1)\wedge v(T_2)\wedge v(T_3)|\gtrsim \theta\}|\gtrapprox|\tubes_{Y^\prime}(Q)|^3.
\end{equation}

\item\label{directionsThetaTubes} Every tube in $\tubes_\theta$ points in a direction from $\Omega$. For each $v\in\Omega$, there are $\sim|\tubes_\theta|/|\Omega|$ tubes from $\tubes_\theta$ that point in direction $v$.

\item\label{thinTubesInsideThetaTubes} Each tube $T_\theta \in\tubes_\theta$ contains $\sim |\tubes^\prime|/|\tubes_\theta|$ tubes from $\tubes^\prime$, and each tube from $\tubes^\prime$ is contained in exactly one tube from $\tubes_\theta$.

\item\label{fineShadingInsideCourseShading} If $T\in\tubes, T_\theta\in\tubes_\theta$, and $T\subset T_\theta$, then $Y^\prime(T)\subset Y_\theta(T_\theta)$.

\item\label{lowerBoundLambdaYTheta} $\lambda_{Y_\theta}\gtrapprox \lambda_{Y^\prime}$. 

\item\label{thetaTubesAreTwoEnds} $(\tubes_\theta,Y_\theta)$ is $(\eps_1,C_1^\prime)$-two-ends, with $C_1^\prime\lessapprox 1$.

\item\label{tubesHaveUniformMult} For each $\delta$-cube $Q\in\mathcal{Q}(Y)$, we have $|\tubes_{Y^\prime}(Q)|\sim \mu_{Y^\prime}$. For each $\theta$-cube $Q\in\mathcal{Q}(Y_\theta)$ we have $|\tubes_{\theta}(Q)|\sim \mu_{Y_\theta}$.

\item\label{innerTubesHaveUniformMult} There is a number $\mu_{\operatorname{fine}}$ so that for each $T_\theta\in\tubes_{\theta}$, we have $\mu_{Y_{T_\theta}}\sim\mu_{\operatorname{fine}}$. 

\item\label{thetaTubesInPlanes} For each $\theta$-cube $Q_\theta\in\mathcal{Q}(Y_\theta)$, there are planes $\Pi_1(Q_\theta),\ldots,\Pi_{B_1}(Q_\theta)$ and collections of tubes $\tubes_{1,\theta}(Q_\theta),\ldots,\tubes_{B_1,\theta} \subset\tubes_\theta(Q_\theta)$ so that for each index $i$, the tubes in $\tubes_{i,\theta}(Q_\theta)$ satisfy $\angle(T_\theta,\Pi_i(Q_\theta))\leq\theta$ and $|\tubes_{i,\theta}|=\mu_{Y_\theta}/B$. Each tube $T\in\tubes_{\theta}(Q)$ is contained in $B_2$ of sets $\{\tubes_{i,\theta}(Q_\theta)\}.$ 

\item\label{productMultiplicity} We have the multiplicity bound
\begin{equation*}
\mu_{Y^\prime}\lessapprox \mu_{Y_{\theta}}\mu_{\operatorname{fine}}/B.
\end{equation*}

\end{enumerate}
\end{lem}
\begin{proof}
Since $(\tubes,Y)$ is $(\epsilon_2,C_2)$-robustly transverse, there exists a number $\delta\leq\theta\leq 1$ (we can suppose that $\theta$ is an integer multiple of $\delta$) and a refinement $(\tubes_1,Y_1)$ of $(\tubes,Y)$ so that $|\tubes_{Y_1}(Q)|\sim\mu_{Y_1}$ for all $Q\in\mathcal{Q}(Y_1)$, and Items \ref{deltaCubePlane} and \ref{mostTriplesWedgeThetaItem} hold for $(\tubes_1,Y_1)$. Note that Item \ref{deltaCubePlane} will continue to hold for any $|\log\delta|^{-O(1)}$-refinement $(\tubes_1^\prime,Y_1^\prime)$ of $(\tubes_1,Y_1)$, and Item \ref{mostTriplesWedgeThetaItem} will hold for every cube $Q\in\mathcal{Q}(Y_1^\prime)$ with $|\tubes_{Y_1^\prime}(Q)|\sim\mu_{Y_1^\prime}.$

Let $\tubes_{\theta,1}$ be a set of essentially distinct $\theta$-tubes so that the following holds.
\begin{itemize}
\item Each tube $T\in\tubes_1$ is contained in at most one $\theta$-tube from $\tubes_{\theta,1}$.
\item
\begin{equation*}
\sum_{T_\theta\in\tubes_{\theta,1}}\sum_{\substack{T\in\tubes_1\\ T\subset T_\theta}}|Y_1(T)|\gtrsim \sum_{T\in\tubes_1}|Y_1(T)|.
\end{equation*}
\end{itemize}
For each $T_\theta\in\tubes_{\theta,1},$ define 
\begin{equation*}
\tubes_1(T_\theta)=\{T\in\tubes_1\colon T\subset T_\theta\}.
\end{equation*}
Let 
\begin{equation*}
\tubes_2=\bigcup_{T\in\tubes_\theta}\tubes_1(T_\theta).
\end{equation*}
For each $T\in\tubes_2$, let $Y_2(T)\subset Y_1(T)$ be a shading so that $(\tubes_2,Y_2)$ is a refinement of $(\tubes_1,Y_1)$, and there exists a number $M$ so that for each $T_\theta\in\tubes_{\theta,1}$ and each $\theta$-cube $Q_\theta$ intersecting $T_\theta$, we have that either $Q_\theta\cap \bigcup_{T\in\tubes_2(T_\theta)}Y_2(T)$ is empty, or 
\begin{equation}\label{numberOfSmallCubesHittingQTheta}
\sum_{T\in\tubes_2(Y_\theta)} |\{Q\colon Q\subset Y_2(T)\cap Q_\theta\} |\sim M.
\end{equation}
For each $T_\theta\in\tubes_{\theta,1}$, define $Y_{1,\theta}$ to be the union of those $\theta$-cubes $Q_\theta$ that intersect $T_\theta$ for which \eqref{numberOfSmallCubesHittingQTheta} holds. Let $\tubes_{\theta}\subset \tubes_{\theta,1}$ so that $(\tubes_{\theta},Y_{\theta})$ is a refinement of $(\tubes_{\theta,1},Y_{\theta,1})$ (here the shading $Y_{\theta}$ is just the restriction of $Y_{\theta,1}$ to the tubes in $\tubes_{\theta}$), and the following properties hold.
\begin{itemize}
\item 
\begin{equation}\label{allTThetaTheSame}
|Y_{\theta}(T_\theta)|/|T_\theta|\sim \lambda_{Y_{\theta}}\quad\textrm{for each}\ T_\theta\in\tubes_{\theta}.
\end{equation}
\item $|\tubes_2(T_\theta)|$ has approximately the same size (up to a factor of two) for each $T_\theta\in\tubes_{\theta}$.
\item There is a set $\Omega\subset S^3$ of $\theta$-separated points so that each tube $T_\theta\in \tubes_{\theta}$ points in a direction from $\Omega$, and there are $\sim |\tubes_{\theta}|/|\Omega|$ tubes from $\tubes_{\theta}$ pointing in each direction.
\end{itemize}

Define $\tubes_3=\bigcup_{T_\theta\in\tubes_{\theta}}\tubes_2(T_\theta)$ and define $Y_3(T)=Y_2(T)$ for each $T\in\tubes_3$. Note that $(\tubes_3,Y_3)$ is a refinement of $(\tubes_2,Y_2)$ (so $(\tubes_3,Y_3)$ is $(\eps_1,C_1^\prime)$-two ends for some $C_1^\prime\lessapprox C_1$), and that \eqref{numberOfSmallCubesHittingQTheta} continues to hold with $(\tubes_3,Y_3)$ in place of $(\tubes_2,Y_2)$.

In particular, there exists at least one tube $T\in\tubes_3$ with $|Y(T)|/|T|\geq\lambda_{Y_3}\gtrapprox\lambda_Y$. But since $Y(T)\subset Y_{\theta}(T_\theta)$ for some tube $T_\theta\in \tubes_{\theta}$, this implies $|Y_{\theta}|/|T_\theta|\geq |Y(T)|/|T| \geq \lambda_{Y_3}$. By \eqref{allTThetaTheSame}, this implies $\lambda_{Y_{\theta}}\gtrsim \lambda_{Y_3}$. Since $\lambda_{Y^\prime}$ (to be defined below) will satisfy $\lambda_{Y_3}\approx\lambda_{Y^\prime}$, this will establish Item \ref{lowerBoundLambdaYTheta}.

At this point, $(\tubes_3,Y_3)$ satisfies Items \ref{deltaCubePlane}; $(\tubes_{\theta}, Y_{\theta})$ satisfies Items \ref{directionsThetaTubes} and \ref{thinTubesInsideThetaTubes}; and the pair $(\tubes_3,Y_3)$ and $(\tubes_{\theta}, Y_{\theta})$ satisfies Item \ref{fineShadingInsideCourseShading}.
 
Observe that 
\begin{equation*}
\begin{split}
\lambda_{Y_{\theta}}&=\frac{1}{|\tubes_{\theta}|}\sum_{T_\theta\in\tubes_{\theta}}|Y_{\theta}(T_\theta)|/|T_\theta|\\
& \sim \frac{1}{\theta^3|\tubes_{\theta}|}\sum_{T_\theta\in\tubes_{\theta}}\sum_{Q_\theta\subset Y_{\theta}(T_\theta)}|Q_\theta|\\
& \sim \frac{1}{\theta^3|\tubes_{\theta}|}\sum_{T_\theta\in\tubes_{\theta,2}}\sum_{Q_\theta\subset Y_{\theta,2}(T_\theta)} M^{-1}(\theta/\delta)^4 \sum_{T\in\tubes_2(T_\theta)}|Q_\theta\cap Y_3(T)|\\
& \sim \frac{\theta}{\delta^4M|\tubes_{\theta}|}\sum_{T\in\tubes_3}|Y_3(T)|\\
& = (\theta/\delta) M^{-1} \frac{1}{|\tubes_{\theta}|}  \sum_{T\in\tubes_3}|Y_3(T)|/|T|\\
& = (\theta/\delta) M^{-1} \frac{|\tubes_3|}{|\tubes_{\theta}|}  \lambda_{Y_3}.
\end{split}
\end{equation*}

Re-arranging, we have
\begin{equation}\label{lambdaY3VsYTheta}
\lambda_{Y_3}\sim (\delta/\theta)M  \frac{|\tubes_{\theta}|}{|\tubes_3|}\lambda_{Y_{\theta}}.
\end{equation}
 
Next we will show that $(\tubes_{\theta},Y_{\theta})$ is $(\eps_1,C_1^\prime)$-two-ends, where $C_1^\prime\lessapprox C_1$. Indeed, let $T_\theta\in\tubes_{\theta}$ and let $B(x,r)\subset\RR^4$ be a ball of radius $r$. We have
\begin{equation}
\begin{split}
|\{Q_\theta\colon &  Q_{\theta}\subset Y_{\theta}(T_\theta)\cap B(x,r)\}|\\
&\sim M^{-1} \sum_{T\in\tubes_3(T_\theta)}|\{Q\colon Q\subset Y_3(T)\cap B(x,r)\}|\\
& \lessapprox M^{-1}|\tubes_3(T_\theta)| r^{\eps}C_1\lambda_{Y_3}\delta^{-1}\\
&\sim M^{-1} \frac{|\tubes_3|}{|\tubes_{\theta,2}|} r^{\eps}C_1\lambda_{Y_3}\delta^{-1}\\
&\sim r^{\eps_1}C_1 \lambda_{Y_3},
\end{split}
\end{equation}
where on the final line we used \eqref{lambdaY3VsYTheta}. We conclude that $(\tubes_{\theta},Y_{\theta})$ satisfies Item \ref{thetaTubesAreTwoEnds}. 

Let $(\tubes^\prime,Y^\prime)$ be a refinement of $(\tubes_3,Y_3)$ so that there is a number $\mu_{\operatorname{fine}}$ so that for each $T_\theta\in\tubes_{\theta}$, we have $\mu_{Y_{T_\theta}}\sim\mu_{\operatorname{fine}}$, and for each $Q\in\mathcal{Q}(Y^\prime)$, we have $|\tubes^\prime(Q)|\sim \mu_{Y^\prime}$. We conclude that $(\tubes^\prime,Y^\prime)$ and $(\tubes_\theta,Y_{\theta})$ satisfy Items \ref{deltaCubePlane} through \ref{innerTubesHaveUniformMult}

Observe that for each $Q\in  \mathcal{Q}(Y^\prime)$, there are $\sim \mu_{Y^\prime}/\mu_{\operatorname{fine}}$ tubes $T_\theta\in\tubes_\theta$ with $Q\in \mathcal{Q}(Y_{T_\theta})$. In particular, if $Q_{\theta}$ is the $\theta$-cube containing $Q$, then 

\begin{equation}\label{numberOfTubesInPlane}
|\{T_\theta\in\tubes_{\theta}(Q_\theta)\colon \angle(\pi(Q),V(T_\theta))\leq\theta\}|\gtrsim \mu_{Y^\prime}/\mu_{\operatorname{fine}}.
\end{equation}

Thus for each $\theta$ cube $Q_\theta\in\mathcal{Q}(Y_\theta)$, we can select sets $\tubes_{1,\theta}(Q_\theta),\ldots,\tubes_{B^\prime,\theta}(Q_\theta)$ that satisfy Item \ref{thetaTubesInPlanes}, with some $B\leq \mu_{Y^\prime}/\mu_{\operatorname{fine}}$. Inequality \eqref{numberOfTubesInPlane} implies that that $(\tubes^\prime,Y^\prime)$ and $(\tubes_\theta,Y_{\theta})$ satisfy Item \ref{productMultiplicity}.
\end{proof}

\subsection{The planebrush argument for weakly plany tubes}

We can now state Proposition \ref{volumeBdPlainyTubes}, which is the main result of this section. Note that hypotheses of Proposition \ref{volumeBdPlainyTubes} have been chosen to match the conclusions of Lemma \ref{planyOrTrilinearLem}.

\begin{prop}\label{volumeBdPlainyTubes}
Let $0<\theta<1$ and let $0<\eps_1<1$. Let $\Omega\subset S^3$ be a set of $\theta$-separated directions. Let $(\tubes,Y)$ be a set of essentially distinct plany $\theta$-tubes and their associated shading. Suppose that 
\begin{itemize}
\item There are $\sim|\tubes|/|\Omega|$ tubes from $\tubes$ pointing in each direction $v\in\Omega$. 
\item $(\tubes,Y)$ is $(\eps_1,C_1)$-two-ends.
\item There are numbers $B,B_1,B_2$ with $B_1\geq B_2$ and $B_1/B_2 \leq B$ so that for each $Q\in\mathcal{Q}(Y)$, there are planes $\Pi_1(Q),\ldots,\Pi_{B_1}(Q)$ and collections of tubes $\tubes_{1}(Q),\ldots,\tubes_{B_1}(Q) \subset\tubes(Q)$ so that for each index $i$, the tubes in $\tubes_{i}(Q)$ satisfy $\angle(T,\Pi_i(Q))\leq\theta$ and $|\tubes_{i}|=\mu_{Y}/B$. Each tube $T\in\tubes(Q)$ is contained in $B_2$ sets from $\{\tubes_{i}(Q)\}.$ 
\end{itemize}
Then for each $\eps>0$, we have

\begin{equation}
\Big|\bigcup_{T\in\tubes}Y(T)\Big|\geq c C_1^{-C/\eps_1}\theta^{\eps}\lambda_Y^{4/3} \theta^{2/3} B^{-4/3}\big(\theta^3|\Omega|\big)^{1/3}\big(\theta^3|\tubes|\big)^{2/3}.
\end{equation}
Here $C$ is an absolute constant and $c>0$ depends on $\eps$ and $\eps_1$.
\end{prop}

\subsection{Reduction to the strongly plany case}
If $(\tubes,Y)$ is a set of $\theta$-tubes and their associated shading that satisfy the hypotheses of Proposition \ref{volumeBdPlainyTubes}, then there exists a $\sim B^{-1}$-refinement $(\tubes^\prime,Y^\prime)$ of $(\tubes,Y)$ so that for each $Q\in\mathcal{Q}(Y^\prime)$, the tubes in $\tubes_{Y^\prime}(Q)$ all lie in the $\theta$-neighborhood of a plane. Indeed, for each cube $Q\in\mathcal{Q}(Y)$, simply select the plane $\Pi(Q)$ whose $\theta$ neighborhood contains the largest number of tubes from $\tubes_{Y}(Q)$. However, if the refinement $(\tubes^\prime,Y^\prime)$ is selected in this way, then it is possible that $(\tubes^\prime,Y^\prime)$ will no longer be $(\eps_1,C_1)$-two-ends. In this section, we will show that it is possible to select the refinement $(\tubes,Y^\prime)$ a bit more carefully and preserve the property of being two-ends.

\begin{lem}\label{twoEndsIntervalSample}
Let $\eps>0$. Let $A\subset[N]=\{1,\ldots,N\}$ with $|A|>N^{\eps}$. Let $C\geq 1$ and let $I\subset N$ be an interval satisfying
\begin{equation}
|A\cap I|\leq T.
\end{equation}

Let $T^{-1}\leq p\leq 1$. Let $A^\prime\subset A$ be obtained by randomly selecting each element of $A$ independently with probability $p$. Then
\begin{equation*}
\operatorname{Pr}\Big(|A^\prime\cap I|\geq 4 (\log N)^{10} pT \Big)\lesssim N^{-10}.
\end{equation*}
\end{lem}
\begin{proof}
First, observe that the expected value of $|A^\prime\cap I|$ is $p|A\cap I|$. By the multiplicative form of Chernoff's bound, we have that for each $t>0$,
\begin{equation}\label{cherfnoffMult}
\operatorname{Pr}\Big(|A^\prime\cap I|\geq (1+t)p|A\cap I|\Big)\leq \exp\big( -t p|A\cap I|/3 \big).
\end{equation}
Applying \eqref{cherfnoffMult} with
\begin{equation*}
t = \frac{3 (\log N)^{10} pT }{p|A\cap I|}=\frac{3(\log N)^{10} T}{|A\cap I|},
\end{equation*}
and noting that since $t\geq 1$,
\begin{equation*}
(t+1)p|A\cap I| \leq 4 (\log N)^{10} pT ,
\end{equation*}
we obtain
\begin{align*}
\operatorname{Pr}\Big(|A^\prime\cap I|\geq 4 (\log N)^{10} pT \Big) & \leq \exp\Big( -\frac{3 (\log N)^{10} T}{|A\cap I|}\big(p|A\cap I|/3\big) \Big)\\
&\leq \exp\Big( - (\log N)^{10}  \big[pT\big] \Big)\\
&\leq \exp\big(- (\log N)^{10} \big)\\
&=N^{-10}.\qedhere
\end{align*}
\end{proof}

\begin{cor}\label{twoEndsRandomRefinement}
Let $\eps>0$. Let $A\subset[N]=\{1,\ldots,N\}$ with $|A|>N^{\eps}$. Suppose that for each interval $I\subset N$ we have
\begin{equation}
|A\cap I|\leq (|I|/N)^{\eps}M.
\end{equation}

Let $M^{-1}\leq p\leq 1$. Let $A^\prime\subset A$ be obtained by randomly selecting each element of $A$ independently with probability $p$. Then
\begin{equation*}
\operatorname{Pr}\Big(|A^\prime\cap I|\geq 4 (\log N)^{10} p (|I|/N)^{\eps} M \quad\ \textrm{for some interval}\ I\subset[N]\Big)\lesssim N^{-8}.
\end{equation*}
\end{cor}
\begin{proof}
Observe that there are $\leq N^2$ intervals $I\subset [N]$. We apply Lemma \ref{twoEndsIntervalSample} to each of these intervals with $T = (|I|/N)^{\eps}M$, and use the union bound.
\end{proof}

To conclude this section, we will show that in order to prove Proposition \ref{volumeBdPlainyTubes}, it suffices to consider the special case where $B=1$. More concretely, it suffices to prove the following result

\begin{prop}\label{volumeBdPlainyTubesB1}
Let $0<\theta<1$ and let $0<\eps_1<1$. Let $\Omega\subset S^3$ be a set of $\theta$-separated directions. Let $(\tubes,Y)$ be a set of essentially distinct plany $\theta$-tubes and their associated shading. Suppose that 
\begin{itemize}
\item There are $\sim|\tubes|/|\Omega|$ tubes from $\tubes$ pointing in each direction $v\in\Omega$. 
\item $(\tubes,Y)$ is $(\eps_1,C_1)$-two-ends.
\item For each $Q\in\mathcal{Q}(Y)$, there exist a $2$-plane $\Pi(Q)$ so that for each $T\in\tubes_Y(Q)$ we have $\angle(v(T),\Pi(Q))\leq\theta$.
\end{itemize}
Then for each $\eps>0$, we have

\begin{equation}\label{volumeBoundPlainyTubes}
\Big|\bigcup_{T\in\tubes}Y(T)\Big|\geq c C_1^{-C/\eps_1}\theta^{\eps}\lambda_Y^{4/3} \theta^{2/3} \big(\theta^3|\Omega|\big)^{1/3}\big(\theta^3|\tubes|\big)^{2/3}.
\end{equation}
Here $C$ is an absolute constant and $c>0$ depends on $\eps$ and $\eps_1$.
\end{prop}

\begin{proof}[Proof of Proposition \ref{volumeBdPlainyTubes} using Proposition \ref{volumeBdPlainyTubesB1}]
Let $(\tubes,Y)$ be a set of $\theta$-tubes and their associated shadings that satisfy the hypotheses of Proposition \ref{volumeBdPlainyTubes}. Let $\eps>0$.

For each $Q\in\mathcal{Q}$, randomly select an index $i_Q\in \{1,\ldots,B_1\}$ uniformly at random. For the remainder of this proof, ``probability'' will be with respect to the random selection of indices $i_Q$ as $Q$ ranges over the elements of $\mathcal{Q}$ (this random selection is equivalent to selecting an element of $\{1,\ldots,B_1\}^{|\mathcal{Q}|}$ uniformly at random). 

Define $\tilde Y(T)\subset Y(T)$ to be the union of all cubes $Q\in \mathcal{Q}^\prime, Q\subset Y(T)$ for which $T\in \tubes_{i}(Q)$ (note that $\tilde Y(T)$ is a random set). Since for each cube $Q\in\mathcal{Q}$, each tube $T\in\tubes_Y(Q)$ is contained in $B_2$ of the sets $\tubes_{i}(Q)$, the random set $\tilde Y (T)$ has the same distribution as the random set obtained by selecting each cube $Q\in \mathcal{Q}, Q\subset Y(T)$ independently with probability $p = B_2/B_1\geq B^{-1}$.

By Corollary \ref{twoEndsRandomRefinement} we have that each $T\in\tubes$ satisfies
\begin{equation}\label{twoEndsYPrime}
\operatorname{Pr}\Big( |\tilde Y(T)\cap B(x,r)|\leq 4r^{\eps_1}((\log\theta)^{10}p ) C_1\lambda_Y |T|\quad\textrm{for all balls}\ B(x,r)\Big)\geq 1-\theta^8.
\end{equation}
An application of Chernoff's bound shows that 
\begin{equation}\label{sizeBoundYPrime}
\operatorname{Pr}\Big( |\log \theta|^{-8}p  |Y(T)| \leq |\tilde Y(T)|\leq |\log \theta|^{8}p  |Y(T)|\Big)\geq 1-\theta^8.
\end{equation}

Since the tubes in $\tubes$ are essentially distinct, we have $|\tubes|\leq\theta^{-6}$, and thus the probability that every tube in $\tubes$ satisfies the events in \eqref{twoEndsYPrime} and \eqref{sizeBoundYPrime} is at least $1-\theta^2$. In particular, there exists a choice of indices $\{i_Q\colon Q\in\mathcal{Q}\}$ so that the events in \eqref{twoEndsYPrime} and \eqref{sizeBoundYPrime} hold for every tube $T\in\tubes$. Fix one such choice of indices, and define $Y^\prime(T)=\tilde Y(T)$.

By \eqref{sizeBoundYPrime}, we have
\begin{equation}\label{boundOnLambdaPrime}
(\log\theta)^{-8}\lambda_Y p \leq \lambda_{Y^\prime} \leq (\log \theta)^8\lambda_Y p.
\end{equation}

By \eqref{boundOnLambdaPrime} and \eqref{twoEndsYPrime}, we have that the pair $(\tubes,Y^\prime)$ is $(\eps_1, 4C_1 (\log\theta)^{16})$-two-ends. The pair $(\tubes^\prime,Y^\prime)$ satisfies the hypotheses of Proposition \ref{volumeBdPlainyTubesB1}. Applying Proposition \ref{volumeBdPlainyTubesB1} with $\eps/2$ in place of $\eps$, we conclude that there is a constant $c^\prime>0$ so that
\begin{align*}
\Big|\bigcup_{T\in\tubes} Y(T)\Big| & \geq \Big| \bigcup_{T\in\tubes}Y^\prime(T)\Big|\\
&\geq c^\prime \big(4C_1 |\log\theta|^{16}\big)^{-C/\eps_1}\theta^{\eps/2}\lambda_{Y^\prime}^{4/3} \theta^{2/3} (\theta^3|\Omega|)^{1/3}(\theta^3|\tubes^\prime|)^{2/3}\\
&\gtrapprox c^\prime  C_1^{-C/\eps_1} \theta^{\eps/2} \lambda_{Y^\prime}^{4/3} p^{-4/3}\theta^{2/3} (\theta^3|\Omega|)^{1/3}(\theta^3|\tubes^\prime|)^{2/3}\\
&\geq c^\prime  C_1^{-C/\eps_1} \theta^{\eps/2} \lambda_{Y^\prime}^{4/3} B^{-4/3}\theta^{2/3} (\theta^3|\Omega|)^{1/3}(\theta^3|\tubes^\prime|)^{2/3}.
\end{align*}
Thus if $c>0$ is chosen sufficiently small, then \eqref{volumeBoundPlainyTubes} holds.
\end{proof}

\subsection{Volume bounds for unions of strongly plany tubes}
Recall that at this point, we have proved Proposition \ref{volumeBdPlainyTubesSpecialCase}, and we have also proved that Proposition \ref{volumeBdPlainyTubesB1} implies Proposition \ref{volumeBdPlainyTubes}. All that remains is to show that Proposition \ref{volumeBdPlainyTubesSpecialCase} implies Proposition \ref{volumeBdPlainyTubesB1}. However, Proposition \ref{volumeBdPlainyTubesB1} is essentially identical to Proposition \ref{volumeBdPlainyTubesSpecialCase} except that the requirement that  $(\tubes,Y)$ be $(\eps_2,C_2)$-robustly transverse has been removed. This is accomplished through a ``robust transversality reduction'' argument nearly identical to that in Proposition \ref{robustTransReduction}. Since the details are nearly identical, we omit them here.

\section{A maximal function estimate in $\RR^4$}\label{maxmlFnEstimateSection}
In this section we will prove Theorem \ref{maximalFunctionEstimateBound}. We will begin with a lemma that lets us upgrade certain assertions of the form $\mathbf{TE}(d,a,b)$ to stronger assertions $\mathbf{TE}(d^\prime,a^\prime,b^\prime)$. Theorem \ref{maximalFunctionEstimateBound} will eventually be proved by iterating this lemma.
 
\begin{lem}\label{boundWithCorrectTubesDependence}
Suppose that Assertion $\mathbf{TE}(4-\alpha,75/28, 1-\alpha/3)$ is true for some $3/4\leq\alpha\leq 1$. Define
\begin{equation}\label{defnOfAlphapPp}
\alpha^\prime = \frac{-118\alpha^2+121\alpha+189}{378-182\alpha},\quad \alpha^{\prime\prime}=\frac{159}{100}-\frac{63}{100\alpha}.
\end{equation}
Then Assertion $\mathbf{TE}(4-\alpha^{\prime\prime},13/4, 1)$ is true. If $\alpha'\leq\alpha,$ then Assertion $\mathbf{TE}(4-\alpha^\prime,75/28, 1-\alpha^\prime/3)$ is true.
\end{lem}
\begin{rem}
Both the hypotheses and conclusions Lemma \ref{boundWithCorrectTubesDependence} contain some slightly strange numerology, which we will remark upon here. First, the requirement that $\alpha\leq 1$ is harmless, since Theorem \ref{wolffThm} states that $RT(3,2,1/2)$ is true. The requirement that $\alpha\geq 3/4$ comes from Theorem \ref{GZTrilinearProp}, which establishes an analogue of the estimate $RT(13/4, 13/4, 1/4)$ under the additional hypothesis that the tubes in question behave in a ``trilinear'' fashion. If this hypothesis is met, then we cannot improve on the bound from Theorem \ref{GZTrilinearProp}.

The number $a=75/28$ describing the $\lambda$ dependence in the hypotheses and conclusions of Lemma \ref{boundWithCorrectTubesDependence} is not particularly important. This is because our argument seeks to find the optimal value of $\alpha^\prime$ and $\alpha^{\prime\prime}$. The corresponding value of $a$ must satisfy several slightly complicated constraints. Rather than tracking these constraints throughout the argument, we have simplified the exposition slightly by choosing a particular value of $a$ satisfying these constraints (the particular value of $a$ chosen here was obtained by beginning with Inequality \eqref{ineqInvolvingPowerLambda} and working backwards). Finally, the values of $\alpha^\prime$ and $\alpha^{\prime\prime}$ from \eqref{defnOfAlphapPp} arise as the natural output of the arguments presented below.
\end{rem}
\begin{proof}
To obtain the above assertions, it suffices (by Proposition \ref{robustTransReduction}) to prove that Assertion $\mathbf{RT}(4-\alpha^{\prime\prime},4-\alpha^{\prime\prime},1)$ and Assertion $\mathbf{RT}(4-\alpha^\prime,75/28, 1-\alpha/3)$ is true. For the latter, note that $\frac{(4-\alpha^\prime)-1}{3}= 1-\alpha'/3\leq 1$, so the hypotheses of Proposition \ref{robustTransReduction} are met. To this end, let $(\tubes,Y)$ be a set of $\delta$-tubes and their associated shading that is $(\eps_1,C_1)$-two-ends and $(\eps_2,100)$-robustly transverse, and let $\eps>0$. 

Apply Lemma \ref{planyOrTrilinearLem} to $(\tubes,Y)$ and let $\theta$, $\Omega$, $(\tubes^\prime,Y^\prime)$, $(\tubes_\theta,Y_\theta)$, and $\mu_{\operatorname{fine}}$ be the output from that Lemma. By Item \ref{tubesHaveUniformMult} from Lemma \ref{planyOrTrilinearLem}, we have that each cube $Q\in\mathcal{Q}(Y^\prime)$ satisfies $|\tubes_{Y^\prime}(Q)|\lesssim \mu_{Y^\prime}$. Thus
\begin{equation}\label{volumeBoundVsMultBound}
\begin{split}
\Big|\bigcup_{T\in\tubes}Y(T)| & \geq \Big|\bigcup_{T\in\tubes^\prime}Y^\prime(T)|\Big|\\
&\gtrsim \mu_{Y^\prime}^{-1}\lambda_{Y^\prime}(\delta^3|\tubes^\prime|)\\
&\gtrapprox \mu_{Y^\prime}^{-1}\lambda_{Y}(\delta^3|\tubes|),
\end{split}
\end{equation}
where on the final line we used the fact that $(\tubes^\prime,Y^\prime)$ is a refinement of $(\tubes,Y)$; this implies that $\lambda_{Y^\prime}(\delta^3|\tubes^\prime|)\gtrapprox \lambda_{Y}(\delta^3|\tubes|)$, and that 
\begin{equation}\label{lambdaYvsYPrime}
\lambda_Y\lessapprox \lambda_{Y^\prime}.
\end{equation} 
Thus it suffices to prove that
\begin{equation}\label{pointwiseMultBoundWrongTubeDependence}
\mu_{Y^\prime}\lessapprox \delta^{-\eps/2} C_1^{C/\eps_1}\min\Big(\lambda_{Y^\prime}^{-\frac{47}{28}}\delta^{-\alpha^\prime}(\delta^3|\tubes|)^{\alpha/3} ,\ \ \lambda_{Y^\prime}^{-9/4}\delta^{-\alpha^{\prime\prime}}\Big),
\end{equation}
where the constant $C$ may depend on $\alpha$, and the implicit constant may depend on $\eps$, $\eps_1,\eps_2$, and $\alpha$. Indeed, combining \eqref{volumeBoundVsMultBound}, \eqref{lambdaYvsYPrime}, and \eqref{pointwiseMultBoundWrongTubeDependence}, we obtain the volume bound
\begin{equation}
\Big|\bigcup_{T\in\tubes}Y(T)\Big|\gtrapprox \delta^{\eps/2}C_1^{-C/\eps_1}\max\Big(\lambda_Y^{75/28}\delta^{\alpha^\prime}(\delta^3|\tubes^\prime|)^{1-\alpha/3},\ \lambda_Y^{4-\alpha^{\prime\prime}}\delta^{\alpha^{\prime\prime}}(\delta^3|\tubes^\prime|)\Big).
\end{equation}
This volume bound is precisely what is needed to establish Assertion $\mathbf{RT}(4-\alpha^{\prime\prime},4-\alpha^{\prime\prime},1)$, and it also establishes Assertion $\mathbf{RT}(4-\alpha^\prime,75/28, 1-\alpha^\prime/3)$ provided $\alpha'\leq\alpha$. 

The remainder of the proof will be devoted to establishing \eqref{pointwiseMultBoundWrongTubeDependence}. Define $A=|\tubes_\theta|/|\Omega|$ and define $B = \mu_{Y^\prime}/\mu_{\operatorname{fine}}$. From Proposition \ref{volumeBdPlainyTubes}, we have the bound
\begin{equation*}
\mu_{Y_\theta}\lessapprox C_1^{-C/\eps_1} \lambda_{Y_\theta}^{-1/3}\theta^{-2/3-\eps/2}A^{1/3}B^{4/3}.
\end{equation*}
From Corollary \ref{KTCor} we have the bound
\begin{equation*}
\begin{split}
\mu_{Y_\theta}&\lessapprox \lambda_{Y_\theta}^{-2}\theta^{-1-\eps/2}A^{3/4}(\theta^3 |\Omega|)^{1/3}\\
& = \lambda_{Y_\theta}^{-2}\theta^{-1-\eps/2}A^{3/4}(\theta^3 A^{-1}|\tubes_\theta|)^{1/3}
\end{split}
\end{equation*}
Combining these, we obtain the bound
\begin{equation}
\begin{split}
\mu_{Y_\theta}&\lessapprox \Big(C_1^{-C/\eps_1} \lambda_{Y_\theta}^{-1/3}\theta^{-2/3-\eps/2}A^{1/3}B^{4/3}\Big)^{3/4}
\Big( \lambda_{Y_\theta}^{-2}\theta^{-1-\eps/2}A^{3/4}(\theta^3 A^{-1}|\tubes_\theta|)^{1/3}\Big)^{1/4}\\
& \leq C_1^{-C/\eps_1} \lambda_{Y_\theta}^{-3/4}\theta^{-3/4-\eps/2}A^{7/16}B (\theta^3 A^{-1}|\tubes_\theta|)^{1/12}.
\end{split}
\end{equation}

From Items \ref{lowerBoundLambdaYTheta} and \ref{productMultiplicity} of Lemma \ref{planyOrTrilinearLem}, we have that 
\begin{equation}\label{pointwiseBdCorrectB}
\begin{split}
\mu_{Y^\prime}& \lessapprox \mu_{Y_\theta}\mu_{\operatorname{fine}}/B\\
&\lessapprox C_1^{-C/\eps_1} \lambda_{Y_\theta}^{-3/4}\theta^{-3/4-\eps/2}A^{7/16} (\theta^3 A^{-1}|\tubes_\theta|)^{1/12}\mu_{\operatorname{fine}}\\
&\lessapprox C_1^{-C/\eps_1} \lambda_{Y^\prime}^{-3/4}\theta^{-3/4-\eps/2}A^{7/16} (\theta^3 A^{-1}|\tubes_\theta|)^{1/12}\mu_{\operatorname{fine}}.
\end{split}
\end{equation}
Since the tubes whose shading contains a $\delta$-cube $Q\in\mathcal{Q}(Y^\prime)$ must lie in the $\theta$ neighborhood of a plane, we also have the bound

\begin{equation}\label{pointwiseBoundPlainyTrivial}
\mu_{Y^\prime}\leq \theta^{-1}\mu_{\operatorname{fine}}. 
\end{equation}

Combining \eqref{pointwiseBdCorrectB} and \eqref{pointwiseBoundPlainyTrivial}, we obtain the bound
\begin{equation}\label{pointwiseThetaBound}
\begin{split}
\mu_{Y^\prime}&\lessapprox  \Big(C_1^{-C/\eps_1} \lambda_{Y^\prime}^{-3/4}\theta^{-3/4-\eps/2}A^{7/16} (\theta^3 A^{-1}|\tubes_\theta|)^{1/12}\mu_{\operatorname{fine}}\Big)^{\frac{16\alpha}{21}}
\Big(\theta^{-1}\mu_{\operatorname{fine}}\Big)^{1-\frac{16\alpha}{21}}\\
&\lessapprox C_1^{-C/\eps_1}\lambda_{Y^\prime}^{\frac{-4\alpha}{7}}\theta^{\frac{4\alpha-21}{21}-\eps/2} A^{\frac{17\alpha}{63}} (\theta^3 |\tubes_\theta|)^{\frac{4\alpha}{63}}\mu_{\operatorname{fine}}.
\end{split}
\end{equation}

Since Assertion $\mathbf{TE}_\delta(4-\alpha,75/28, 1-\alpha/3)$ is true, we can apply it to the tubes contained in each of the $\theta$-tubes $T_\theta\in\tubes_\theta$. We conclude that
\begin{equation}\label{muFineEstimateInduction}
\begin{split}
\mu_{\operatorname{fine}}&\lesssim  C_1^{-C/\eps_1}\lambda_{Y^\prime}^{-\frac{47}{28}}(\delta/\theta)^{-\alpha-\eps/2}\Big((\delta/\theta)^3|\tubes|/|\tubes_\theta|\Big)^{\alpha/3}\\
&\lesssim C_1^{-C/\eps_1} \lambda_{Y^\prime}^{-\frac{47}{28}}(\delta/\theta)^{-\alpha-\eps/2}\Big((\delta/\theta)^3|\tubes|/|\tubes_\theta|\Big)^{4\alpha/63}A^{-17\alpha/63},
\end{split}
\end{equation}
where the implicit constant depends on $\eps,\eps_1,\eps_2,$ and $\alpha$. On the second line we used the fact that there are $\sim A$ tubes in $T_\theta$ pointing in each $\theta$-separated direction, and each of these $\theta$-tubes contain $\lesssim A^{-1}(\theta/\delta)^3$ tubes from $\tubes^\prime$. 

Combining \eqref{pointwiseThetaBound} and \eqref{muFineEstimateInduction}, we have 
\begin{equation}\label{boundInvolvingTheta}
\begin{split}
\mu_{Y^\prime}
&\lessapprox \Big(C_1^{-C/\eps_1}\lambda_{Y^\prime}^{\frac{-4\alpha}{7}}\theta^{\frac{4\alpha-21}{21}-\eps/2} A^{\frac{17\alpha}{63}} (\theta^3 |\tubes_\theta|)^{\frac{4\alpha}{63}}\Big) \\
&\quad\cdot\Big(C_1^{-C/\eps_1} \lambda_{Y^\prime}^{-\frac{47}{28}}(\delta/\theta)^{-\alpha-\eps/2}\Big((\delta/\theta)^3|\tubes|/|\tubes_\theta|\Big)^{4\alpha/63}A^{-17\alpha/63}  \Big)\\
&\leq C_1^{-C/\eps_1} \lambda_{Y^\prime}^{-\frac{4\alpha}{7} - \frac{47}{28}}\theta^{\frac{25\alpha-21}{21}}\delta^{-\alpha-\eps/2}(\delta^3|\tubes|)^{4\alpha/63}.
\end{split}
\end{equation}
By Corollary \ref{trilinearVolumeBound}, we also have the estimate
\begin{equation}
\mu_{Y^\prime} \lessapprox  \lambda_{Y^\prime}^{-9/4}\theta^{-1}\delta^{-3/4-\eps/2}(\delta^3|\tubes|)^{3/4}.
\end{equation}
Interpolating these two estimates, we have
\begin{equation}\label{multiplicityBoundWrongPowerOfTubes}
\begin{split}
\mu_{Y^\prime}&\lessapprox \Big(C_1^{-C/\eps_1} \lambda_{Y^\prime}^{-\frac{4\alpha}{7} - \frac{47}{28}}\theta^{\frac{25\alpha-21}{21}}\delta^{-\alpha-\eps/2}(\delta^3|\tubes|)^{4\alpha/63}\Big)^{\frac{21}{25\alpha}} \\
&\quad\cdot\Big(\lambda_{Y^\prime}^{-9/4}\theta^{-1}\delta^{-3/4-\eps/2}(\delta^3|T|)^{3/4}\Big)^{1-\frac{21}{25\alpha}}\\
&\lessapprox \lambda_{Y^\prime}^{\frac{12}{25\alpha}-\frac{273}{100}} \delta^{\frac{63}{100\alpha}-\frac{159}{100}-\eps/2}  (\delta^3|\tubes|)^{\frac{241}{300}-\frac{63}{100\alpha}}.
\end{split}
\end{equation}
Observe that for $3/4\leq\alpha\leq 1$ we have
\begin{equation}\label{ineqInvolvingPowerLambda}
\frac{12}{25\alpha}-\frac{273}{100}\geq -\frac{9}{4}.
\end{equation}
Thus 
\begin{equation}\label{fistMuYpEstimate}
\mu_{Y^\prime}\lessapprox C_1^{C/\eps_1}\delta^{-\eps/2}\lambda_{Y^\prime}^{-9/4}\delta^{-\alpha^{\prime\prime}}.
\end{equation}
This inequality is the second term in \eqref{pointwiseMultBoundWrongTubeDependence}, so we have proved half of the inequality. Our next task is to establish the remaining half of \eqref{pointwiseMultBoundWrongTubeDependence}.

From Theorem \ref{wolffThm}, we have the estimate
\begin{equation}
\mu_{Y^\prime} \lessapprox \lambda_{Y^\prime}^{-1} \delta^{-1} (\delta^3|\tubes|)^{1/2}.
\end{equation}
Interpolating these two estimates, we have

\begin{equation}
\begin{split}
\mu_{Y^\prime}&\lessapprox  \Big( \lambda_{Y^\prime}^{\frac{12}{25\alpha}-\frac{273}{100} } \delta^{\frac{63}{100\alpha}-\frac{159}{100}-\eps/2}  (\delta^3|\tubes|)^{\frac{241}{300}-\frac{63}{100\alpha}} \Big)^{ \frac{50\alpha(3-2\alpha)}{189-91\alpha}}\\
&\quad\cdot
\Big(\lambda_{Y^\prime}^{-1}\delta^{-1}(\delta^3|\tubes|)^{\frac{1}{2}} \Big)^{1-\frac{50\alpha(3-2\alpha)}{189-91\alpha}}\\
&\leq C_1^{-C/\eps_1}\lambda_{Y^\prime}^{\frac{346\alpha^2-433\alpha-234}{14(27-13\alpha)} }  \delta^{\frac{118\alpha^2-121\alpha-189}{378-182\alpha}} (\delta^3|\tubes|)^{\alpha/3}.
\end{split}
\end{equation}
Note that if $3/4\leq\alpha\leq 1$, then
\begin{equation*}
\frac{346\alpha^2-433\alpha-234}{14(27-13\alpha)} \geq -\frac{321}{196} > -\frac{47}{28}.
\end{equation*}
Thus
\begin{equation}\label{secondMuYpEstimate}
\mu_{Y^\prime}\lessapprox C_1^{C/\eps_1} \delta^{-\eps/2}  \lambda_{Y^\prime}^{-\frac{47}{28}}\delta^{-\alpha^{\prime}}(\delta^3|\tubes|)^{\alpha/3}.
\end{equation}
The estimates \eqref{fistMuYpEstimate} and \eqref{secondMuYpEstimate} give us \eqref{pointwiseMultBoundWrongTubeDependence}.
\end{proof}

\begin{lem}\label{TwoEndsBoundCorrectTubeDependence}
Let $\alpha = \frac{1}{128}(257-\sqrt{17665})$, let $d_0 = \dimBoundSymbolic$, and let $d_1=\maxmlBoundSymbolic$. Then for each $\eps>0$, 
\begin{equation}\label{Ass1IsTrue}
\operatorname{Assertion} \mathbf{TE}\big(4-\alpha-\eps,\ \frac{75}{28},\ 1- \alpha/3\big)\quad\textrm{is true},
\end{equation} 
\begin{equation}\label{Ass2IsTrue}
\operatorname{Assertion}\mathbf{TE}\big(d_0-\eps,\ 13/4,\ 0\big)\quad\textrm{is true},
\end{equation} 
and 
\begin{equation}\label{Ass3IsTrue}
\operatorname{Assertion}\mathbf{TE}\big(d_1-\eps,\ d_1+\eps,\ 0\big)\quad\textrm{is true},
\end{equation} 
\end{lem}
\begin{proof}
We will begin with \eqref{Ass1IsTrue}. When $\alpha_1=1$, Assertion $\mathbf{TE}_\delta(4-\alpha_1,\frac{75}{28}, 1-\alpha_1/3)$ is implied by Assertion $\mathbf{TE}_\delta(3, 2, 2/3)$, which follows from Theorem \ref{wolffThm} and Proposition \ref{robustTransReduction}. For each $k\geq 1$ suppose that Assertion $\mathbf{TE}(4-\alpha_{k},\frac{75}{28}, 1-\alpha_k)$ is true, and define 
\begin{equation*}
\alpha_{k+1} = \frac{-118\alpha_k^2+121\alpha_k+189}{378-182\alpha_k}.
\end{equation*}
Since $\alpha_{k+1}\leq\alpha_k$, Lemma \ref{boundWithCorrectTubesDependence} implies that Assertion $\mathbf{TE}(4-\alpha_{k+1},\frac{75}{28}, 1-\alpha_{k+1}/3)$ is true.  Observe that $\alpha_k \searrow \alpha$. Thus for each $\eps>0$, there is an index $k$ so that $\alpha_k\leq \alpha+\eps$. We conclude that \eqref{Ass1IsTrue} holds.

For \eqref{Ass2IsTrue}, note that for each $\eps>0$, \eqref{Ass2IsTrue} follows from Lemma \ref{boundWithCorrectTubesDependence} and \eqref{Ass1IsTrue}. Finally, for \eqref{Ass3IsTrue}, we will take an appropriate geometric mean of the estimates
\begin{equation*}
\Big| \bigcup_{T\in\tubes}Y(T) \Big|\gtrsim \lambda_Y^{\frac{75}{28}}\delta^{\alpha+\eps},\quad \Big| \bigcup_{T\in\tubes}Y(T) \Big|\gtrsim \lambda_Y^{\frac{13}{4}}\delta^{4-d_0+\eps},
\end{equation*}
to obtain an estimate of the form
\begin{equation*}
\Big| \bigcup_{T\in\tubes}Y(T) \Big|\gtrsim \lambda_Y^{d_1}\delta^{4-d_1+\eps}
\end{equation*}
I.e.~we must solve
\[
\Big(4 - \frac{1}{128}\big(256-\sqrt{17665}\big)\Big)\theta + \Big(\dimBoundSymbolic\Big)(1-\theta) = \frac{75}{28}\theta+\frac{13}{4}(1-\theta),
\]
and with this value of $\theta$ we define $d_1$ to satisfy
\[
4-d_1 = (4-\alpha)\theta + d_0(1-\theta).\qedhere
\]
\end{proof}
We are now ready to prove Theorem \ref{maximalFunctionEstimateBound}.
\begin{proof}[Proof of Theorem \ref{maximalFunctionEstimateBound}]
We begin with the maximal function estimate \eqref{maxmlFnBd}. Let $\tubes$ be a set of $\delta$-tubes in $\RR^4$ that point in $\delta$-separated directions and let $\eps>0$. Without loss of generality we can assume that $|\tubes|\sim\delta^{-3}$, since if this assumption fails then we can add additional tubes while maintaining the requirement that the tubes in $\tubes$ point in $\delta$-separated directions. Doing this can only increase $\Big\Vert \sum_{T\in\tubes}\chi_T \Big\Vert_{d/(d-1)}$. 

For each tube $T\in\tubes,$ let $T^\prime$ be the union of all $\delta$-cubes that intersect $T$. In particular, we have $T\subset T^\prime$, and the function $\sum_{T\in\tubes}\chi_{T^\prime}$ is constant on each $\delta$-cube.

Since the function $\sum_{T\in\tubes}\chi_{T^\prime}$ must take integer values between $1$ and $|\tubes|\lesssim \delta^{-3}$, by dyadic pigeonholing we can select an integer $\mu$ and a set $X\subset\RR^4$ that is a union of $\delta$ cubes so that
\begin{equation*}
\sum_{T\in\tubes}\chi_{T^\prime}(x)\sim\mu\quad\textrm{for every}\ x\in X,
\end{equation*}
and
\begin{equation}\label{XmuSize}
|X|^{(d-1)/d}\mu\gtrapprox \Big\Vert\sum_{T\in\tubes}\chi_{T^\prime}\Big\Vert_{d/(d-1)}\geq \Big\Vert\sum_{T\in\tubes}\chi_{T}\Big\Vert_{d/(d-1)}.
\end{equation}
For each $T\in\tubes,$ define the set $Y(T) = T^\prime\cap X$; we have that $Y(T)$ is a union of $\delta$-cubes that intersect $T$, so it is a shading of $T$ in the sense of Definition \ref{shading}. By construction we have $\mu_Y \sim \mu$, and
\begin{equation*}
\sum_{T\in\tubes}|Y(T)|=\mu|X|.
\end{equation*}
In particular, this implies
\begin{equation}\label{lambdaYSim}
\lambda_Y = \frac{\mu|X|}{\delta^3|\tubes|}\sim \mu|X|.
\end{equation}

Inequality \eqref{Ass2IsTrue} from Lemma \ref{TwoEndsBoundCorrectTubeDependence} and Proposition \ref{twoEndsConditionThm} imply that
\begin{equation}\label{sizeXBd}
|X| = \Big|\bigcup_{T\in\tubes}Y(T)\Big|\geq c_{\eps}\lambda^d\delta^{4-d+\eps},\quad d = \maxmlBoundSymbolic,
\end{equation}
and thus
\begin{equation}\label{muBound}
\mu\leq c_{\eps}^{-1}\lambda^{1-d}\delta^{d-4-\eps}.
\end{equation}
Combining \eqref{XmuSize}, \eqref{lambdaYSim}, and \eqref{sizeXBd}, we have
\begin{equation*}
\begin{split}
\Big\Vert\sum_{T\in\tubes}\chi_{T}\Big\Vert_{d/(d-1)} & \lesssim |X|^{\frac{d-1}{d}}\mu\\
&= (|X|\mu)^{\frac{d-1}{d}}\mu^{\frac{1}{d}}\\
&  \lessapprox \lambda_Y^{\frac{d-1}{d}} \Big(c_{\eps}^{-1}\lambda^{1-d}\delta^{d-4-\eps}\Big)^{\frac{1}{d}}\\
& = c_{\eps}^{1/d}\delta^{1-4/d-\eps/d}.
\end{split}
\end{equation*}
This gives \eqref{Ass1IsTrue}.

For our Hausdorff dimension estimate, we argue as above, except we use Inequality \eqref{Ass3IsTrue} from Lemma \ref{TwoEndsBoundCorrectTubeDependence} to obtain the bound
\begin{equation}\label{sizeXBdWorseLambda}
|X| = \Big|\bigcup_{T\in\tubes}Y(T)\Big|\geq c_{\eps}\lambda^{13/4}\delta^{4-d+\eps},\quad d = \dimBoundSymbolic.\qedhere
\end{equation} 
\end{proof}

\appendix

\section{Is the planebrush estimate sharp?}
In this section we will informally explore the question of whether the planebrush estimate is sharp. In \cite{KLT}, \L{}aba, Tao, and the first author considered the Heisenberg group
\begin{equation*}
\mathbb{H}=\{(z_1,z_2,z_3)\in\CC^3\colon \operatorname{Im}(z_3) = \operatorname{Im}(z_1\bar z_2)\}.
\end{equation*}
The closure of $\mathbb{H}\cap B(0,1)$ is a compact subset of $\CC^3$ that has many of the properties of a $5/2$ dimensional counter-example to the Kakeya conjecture in $\RR^3$. In particular, it is a $5/2$ dimensional subset of $\CC^3$ (here we mean the dimension of the set is $5/2$ times the dimension of the underlying field $\CC$) that contains a two (complex) dimensional family of complex lines, and these lines satisfy a natural analogue of the Wolff axioms. Since Wolff's hairbrush arguments from \cite{W} apply equally well to the set $\mathbb{H}\cap B(0,1)$, we say that Wolff's hairbrush arguments cannot distinguish the Heisenberg group from a genuine Besicovitch set. Informally, we say that Wolff's hairbrush argument is sharp, since it cannot be improved without incorporating additional information about the configuration of lines (for example, the fact that the lines point in different directions or that the underlying field does not contain a half-dimensional subfield). 

It is an interesting open question whether there exists a field $F$ and a set $X\subset F^4$ of dimension $3+1/3$ (or perhaps cardinality $|F|^{3+1/3}$ if $F$ is finite) so that $X$ contains a 3 dimensional family of lines satisfying the Wolff axioms, with the property that for each point $x\in X$, the lines passing through $x$ are coplanar (i.e. they are all contained in a common plane). If such a set exists, it would suggest that the planebrush argument from Section \ref{planebrushSection} is ``sharp,'' in the same sense that Wolff's hairbrush argument is sharp.

We hypothesize that if such a set $X\subset F^4$ does exist, then it is likely of the following type. First, the field $F$ is a degree-three field extension of some smaller field $K$, and $X$ is a low-degree 10-dimensional subvariety of $K^{12}$. To date, however, the authors have been unsuccessful in either finding such a set $X\subset F$ or in showing that no such example of this type can exist.

\section{Some Remarks on the Kakeya problem in $\mathbb{R}^3$}\label{R3KakeyaSec}
In \cite{KZ}, the authors proved the following volume estimate for unions of tubes in $\mathbb{R}^3$.
\begin{thm}[\cite{KZ}, Theorem 1.2]\label{GZMainThm}
There exist positive constants $C$ (large) and $c>0,\ \eps_0>0$ (small) so that the following holds. Let $\delta>0,\ \delta\leq\lambda\leq 1$, and let $(\tubes,Y)$ be a set of $\delta$-tubes and their associated shading that satisfy the Wolff axioms. Suppose that $\sum_{T\in\tubes}|Y(T)|\geq\lambda$. Then
\begin{equation}\label{volumeBdEqn}
\Big|\bigcup_{T\in\tubes}Y(T)\Big|\geq c\lambda^C \delta^{1/2-\eps_0}.
\end{equation}
\end{thm}
Theorem \ref{GZMainThm} immediately implies that every Besicovitch set in $\mathbb{R}^3$ has Hausdorff dimension at least $5/2+\eps_0$. Theorem \ref{GZMainThm} does not immediately yield a maximal function estimate, because the exponent of $\lambda$ in \eqref{volumeBdEqn} is wrong---an exponent of $\lambda^{5/2+\eps_0}$ is required in order to obtain a maximal function estimate.

However, the same argument used in Section \ref{maxmlFnEstimateSection} can also be used to upgrade Theorem \ref{GZMainThm} to a maximal function estimate. Indeed, recall the following consequence of Wolff's hairbrush argument in $\mathbb{R}^3$:

\begin{thm}[Wolff Hairbrush Estimate]\label{wolffHairbrushEstimateR3}
There exists an absolute constant $C$ so that the following holds. Let $\delta>0,\ \delta\leq\lambda\leq 1$, $\eps_1>0$, and let $(\tubes,Y)$ be a set of $\delta$-tubes and their associated shading that satisfy the Wolff axioms. Suppose that $\sum_{T\in\tubes}|Y(T)|\geq\lambda$ and that $(\tubes,Y)$ is $(\eps_1,\lambda)$-two-ends. Then for each $\eps>0$, there exists a constant $c_\eps>0$ so that
\begin{equation}\label{volumeBdEqnWolff}
\Big|\bigcup_{T\in\tubes}Y(T)\Big|\geq c_\eps \lambda^{2} \delta^{1/2+\eps+C\eps_1}.
\end{equation}
\end{thm}
The proof of Theorem \ref{wolffHairbrushEstimateR3} is essentially the same as the proof of Theorem \ref{wolffThm}. Observe that an analogue of \eqref{volumeBdEqnWolff} with the exponent $\lambda^{5/2}$ would be sufficient to establish a maximal function estimate in $\RR^3$ at dimension $5/2$. Thus we can interpolate \eqref{volumeBdEqnWolff} (which has a better than necessary $\lambda$ exponent) with \eqref{volumeBdEqn} (which has a worse than necessary $\lambda$ exponent) to obtain an improved maximal function estimate in $\mathbb{R}^3$:

\begin{thm}[Maximal function estimate in $\mathbb{R}^3$]
There exist absolute constants $C$ (large) and $\eps_0>0$ (small) so that the following holds. Let $\delta>0$ and let $\tubes$ be a set of tubes that satisfy the Wolff axioms. Then
\begin{equation}\label{maximalFunctionEstimate}
\Big\Vert \sum_{T\in\tubes}\chi_{T}\Big\Vert_d \leq C \delta^{1-3/d},\quad d = 5/2 + \eps_0. 
\end{equation}
\end{thm}
Note that the constant $\eps_0>0$ is smaller than the corresponding constant in Theorem \ref{GZMainThm}.
\begin{proof}
Using the standard two-ends reduction (see Proposition \ref{twoEndsConditionThm} and the accompanying remark; Proposition \ref{twoEndsConditionThm} is stated for tubes in $\RR^4$, but an analogous statement holds in any dimension), it suffices to prove the following. Let $\eps_1,\eps_2>0$ be sufficiently small. Then there exists $\eps_0>0$ so that the following holds. Let $(\tubes,Y)$ be a set of tubes and their associated shading. Suppose that the tubes satisfy the generalized Wolff axioms; $|\tubes|\geq \delta^{-2+\eps_1}$; $|Y(T)|\sim\lambda$ for each $T\in\tubes$; and $(\tubes,Y)$ is $(\eps_2,\delta^{-1}\lambda)$ two-ends. Then 
\begin{equation}\label{volumeLowerBoundR3}
\Big|\bigcup_{T\in\tubes}Y(T)\Big|\geq c\lambda^{5/2+\eps_0}\delta^{1/2-\eps_0}.
\end{equation}
The estimate \eqref{volumeLowerBoundR3} follows from averaging appropriate powers of \eqref{volumeBdEqn} and \eqref{volumeBdEqnWolff}.
\end{proof}

\bibliographystyle{abbrv}
\bibliography{R4_Grainy_Kakeya}

\end{document}